\documentclass[preprint, reqno]{amsart}
\usepackage{amsfonts, amsmath, amsthm, amssymb, amsrefs}
\usepackage{mathtools, mathrsfs}
\usepackage[scr=esstix]{mathalfa}

\usepackage{enumitem}

\usepackage{hyperref}
\hypersetup{colorlinks={true},linkcolor={blue},citecolor=red}
\usepackage[capitalize]{cleveref}

\DeclareMathOperator*{\supp}{supp}
\DeclareMathOperator*{\diam}{diam}

\renewcommand{\epsilon}{\varepsilon}
\renewcommand{\subset}{\subseteq}

\newcommand{\dist}{\mathrm{dist}}

\newcommand{\lbonpval}{\frac{1}{2}(3+\sqrt{5})}
\newcommand{\card}{\mathrm{card}}

\theoremstyle{definition}
\newtheorem{definition}{Definition}[section]
\newtheorem{theorem}[definition]{Theorem}
\newtheorem{lemma}[definition]{Lemma}
\newtheorem{corollary}[definition]{Corollary}

\newtheorem{remark}[definition]{Remark}
\newtheorem{proposition}[definition]{Proposition}

\AddToHook{env/theorem/begin}{\crefalias{definition}{theorem}}
\AddToHook{env/lemma/begin}{\crefalias{definition}{lemma}}
\AddToHook{env/corollary/begin}{\crefalias{definition}{corollary}}
\AddToHook{env/example/begin}{\crefalias{definition}{example}}
\AddToHook{env/remark/begin}{\crefalias{definition}{remark}}
\AddToHook{env/proposition/begin}{\crefalias{definition}{proposition}}

\numberwithin{equation}{section}

\author[L.\ O'Brien]{Lucas O'Brien}
\address{Lucas O'Brien \\ Department of Mathematics \\ University of
  Toronto \\ Toronto, Canada \\ {Email: lucas.obrien@mail.utoronto.ca}}

\author[F.\ Kobayashi]{Forest Kobayashi}
\address{Forest Kobayashi\\ Department of Mathematics \\ University of
  British Columbia\\ Vancouver, Canada\\ {Email:
    {fkobayashi@math.ubc.ca}}}

\author[Y.H.\ Kim]{Young-Heon Kim}
\address{Young-Heon Kim\\ Department of Mathematics \\ University of
  British Columbia\\ Vancouver, Canada\\ {Email: yhkim@math.ubc.ca}}

 \subjclass[]{} \thanks{ LO is partially supported by the
  Natural Sciences and Engineering Research Council of Canada (NSERC)
  Undergraduate Student Research Award. FK is supported by the
  doctoral fellowship of the University of British Columbia. YHK is
  partially supported by the Natural Sciences and Engineering Research
  Council of Canada (NSERC), with Discovery Grant RGPIN-2019-03926, as
  well as Exploration Grant (NFRFE-2019-00944) from the New Frontiers
  in Research Fund (NFRF). YHK is also a member of the Kantorovich
  Initiative (KI), which is supported by the PIMS Research Network
  (PRN) program of the Pacific Institute for the Mathematical Sciences
  (PIMS). We thank PIMS for their generous support. Part of this work
  was completed during YHK's visit at the Korea Advanced Institute of
  Science and Technology (KAIST), and we thank them for their
  hospitality and the excellent environment. \copyright 2025 by the
  authors. All rights reserved.}

\date{}

\title{Structure of average distance minimizers in general dimensions}

\begin{document}
\begin{abstract}
  For a fixed, compactly supported probability measure $\mu$ on the
  $d$-dimensional space $\mathbb{R}^d$, we consider the problem of
  minimizing the $p^{\mathrm{th}}$-power average distance functional
  over all compact, connected $\Sigma \subseteq \mathbb{R}^d$ with
  Hausdorff 1-measure $\mathcal{H}^1(\Sigma) \leq l$. This problem,
  known as the average distance problem, was first studied by
  Buttazzo, Oudet, and Stepanov in 2002, and has undergone a
  considerable amount of research since. We will provide a novel
  approach to studying this problem by analyzing it using the
  so-called \textit{barycentre field} considered previously by Hayase
  and two of the authors. This allows us to provide a complete
  topological description of minimizers in arbitrary dimensions when
  $p = 2$ and $p > \frac{1}{2}(3 + \sqrt{5}) \approx 2.618$, the first
  such result that includes the case when $d > 2$.
\end{abstract}

\maketitle

\tableofcontents

\section{Introduction}\label{sec:introduction}

Suppose one has been tasked with constructing a network $\Sigma$ of
water pipes for a city in which demand is distributed according to a
given probability measure $\mu$ and the construction costs are
modelled by a given functional $\mathscr{C}(\Sigma)$. Given a fixed
budget $l$, how can one determine the best possible network shape
$\Sigma$?

To be more specific, let $\mu$ be a probability measure on
$\mathbb{R}^d$, that is, $\mu \in \mathcal P(\mathbb R^d)$, where
$d\ge 1$. Suppose that for fixed $p \geq 1$, each ``client'' $x \in
\supp \mu$ has to pay a price $\dist^p(x, \Sigma)$ to connect to the
network $\Sigma$. Then, the best network $\Sigma$ should be the one
minimizing the $\mu$-average cost
\[
  \mathscr J_p(\Sigma) = \int_{\mathbb R^d} \dist^p(x, \Sigma) \,
  \mathrm{d}\mu(x).
\]
The choice of different values for $p \geq 1$ affects the proportional
weighting of outliers in $\dist(x, \Sigma)$, and for us the main case
is $p=2$. The problem may now be stated succinctly as:
\begin{equation}
  \begin{cases}
    \text{Minimize:} \quad
    & \mathscr J_p(\Sigma) \\
    \text{Subject to:} \quad
    & \mathscr C(\Sigma) \leq l.
  \end{cases}
  \label{eq:objective-minimization}
\end{equation}

The choice of constraint functional $\mathscr C$ greatly impacts the
qualitative traits of solutions to \eqref{eq:objective-minimization}
as well as the analysis required to characterize them. Broadly, the
two categories one typically considers are ``intrinsic'' constraints
such as $\mathcal H^1(\Sigma)$ and ``extrinsic'' constraints such as
$\mathrm{arclength}(\gamma)$, where $\gamma : [0,1] \to \mathbb R^d$
parametrizes $\Sigma$. We briefly discuss some references for
different formulations in \cref{sec:related-works}.

In this paper we consider the \emph{average distance problem} (ADP),
also called the \emph{irrigation problem}, which was first introduced
by Buttazzo, Oudet, and Stepanov in 2002 \cite{Buttazzo02}:
\begin{equation}
  \label{eq:adp}
  \begin{cases}
    \text{Minimize:} \quad
    & \mathscr J_p(\Sigma) \\
    \text{Subject to:} \quad
    & \Sigma \text{ is compact and connected, and } \mathcal H^1(\Sigma)
    \leq l.
  \end{cases}
  \tag{ADP}
\end{equation}
For the \eqref{eq:adp}, a longstanding problem of interest---which we
address in this paper---is characterizing the topology of optimizers.
A chronology of progress on the topological characterization problem
can be found in the survey \cite{Lemenant12}; let us briefly summarize
the main points.

The first topological characterization result for the \eqref{eq:adp}
was obtained by Buttazzo and Stepanov in 2003 \cite{Buttazzo03}. In
that work they proved for $d = 2$ that if $\mu \ll \mathrm{Lebesgue}$
and has a density that is $L^q$ for some $q > 4/3$, then optimizers of
the \eqref{eq:adp} are \emph{(topological) finite, binary trees}. That
is, optimizers contain no loops, are comprised of finitely many
``branches'' (homeomorphic copies of intervals), and the branches only
meet in triple junctions. In 2004, Paolini and Stepanov proved the
absence of loops in general dimensions $d \geq 2$ \cite[Theorem
5.6]{Stepanov04} under only a mild regularity assumption on $\mu$;
namely, that if $\Sigma$ is optimal then $\mu(\Sigma) = 0$. However,
the finitely-many branches and triple junction results remained
unknown.

In 2006, Stepanov \cite{Stepanov06} established the following
``conditional'' topological characterization for $d \geq 2$. For fixed
$\Sigma$, let $\pi_\Sigma$ denote any measurable selection of the
closest point projection onto $\Sigma$. Let $\Sigma$ be an optimizer,
and suppose (a) that $\mu(\Sigma) = 0$, and (b) that
\begin{equation}
  \text{ the pushforward measure }\nu \coloneqq (\pi_\Sigma)_{\#} \mu
  \text{ contains an atom.} \label{eq:existence-of-atom}
\end{equation}
Then, by \cite[Theorem 5.5]{Stepanov06}, $\Sigma$ is a finite, binary
tree.

As we do not know the optimizer a priori, Stepanov's hypothesis (a) is
effectively equivalent to the condition that
\begin{equation}\label{eq:regularity-mu}
  \text{$\mu (A) =0$ for each set $A\subset \mathbb{R}^d$ with
    $\mathcal{H}^1(A) < \infty.$}
\end{equation}
The hypothesis \eqref{eq:regularity-mu} is essentially sharp: If
$\supp \mu$ is a 1-d set and $l \geq \mathcal H^1(\supp \mu)$, then
for all optimizers $\Sigma$ one has $\supp \mu \subseteq \Sigma$; thus
if $\supp \mu$ contains at least one of $\{\text{a loop},\
\text{infinitely-many branches},\ \text{more-than-triple junctions}\}$
then $\Sigma$ will too.

In this way, Stepanov's result \cite[Theorem 5.5]{Stepanov06} reduces
the topological characterization question to proving the existence of
an atom of $\nu$. Such existence seems, intuitively, quite plausible,
since ``$\Sigma$ has no loops'' \cite[Theorem 5.6]{Stepanov04} implies
``$\Sigma$ has endpoints,'' and (as noted in the discussion following
\cite[Lemma 7]{Lemenant12}) we would expect those ``endpoints'' to
receive positive mass in $\nu$.

Despite this intuitive picture, in dimensions $d > 2$, the existence
of an atom result has remained elusive in the ensuing years. Hence,
the full topological characterization has remained open---though,
seemingly, tantalizingly close---in dimensions $d > 2$.

\subsection{Key Contributions of Our Paper}

In this work, we resolve the topological characterization problem in
general dimensions $d \geq 2$ for most values of $p \geq 2$, including
$p=2$. This is our first main theorem:
\begin{theorem}[see \cref{topologicalcharacterization} for a more
  precise statement]\label{main:topologicalcharacterization} Let $d\ge
  2$ and let $p = 2$ or $p > \lbonpval$. Assume
  \eqref{eq:regularity-mu} and let $\Sigma_{\rm opt}$ be a solution to
  the \eqref{eq:adp}. Then, (a) $\Sigma_{\rm opt}$ does not contain
  any loops and is thus a topological tree, (b) $\Sigma_{\rm opt}$ has
  only finitely many endpoints and branching points, and (c) each
  branching point is a triple junction.
\end{theorem}

Note that part (a) was already proven in \cite{Stepanov04}*{Theorem 5.6}.
Our approach centers on establishing the missing link that turns
Stepanov's \emph{conditional} result \cite{Stepanov06}*{Theorem 5.5}
into a \emph{general} one: Namely, that for each optimal $\Sigma$, the
measure $\nu$ contains an atom \eqref{eq:existence-of-atom}.

To obtain this result, the key insight is to view the problem from the
perspective of the \emph{barycentre field}
(\cref{def:barycentrefield}), which was considered previously by
Hayase and two of the authors in \cite{Kobayashi24} in studying a
version of \eqref{eq:objective-minimization} where $\mathscr C$ is
defined via a Sobolev norm. Roughly speaking, the barycentre field is
a vector field on $\Sigma$ encoding the ``gradient'' of the objective
$\mathscr J_p(\Sigma)$ under continuous perturbations of $\Sigma$
(\cref{strongbarycentreapproximation}). Similar ideas to the
barycentre field had appeared in the literature on
\eqref{eq:objective-minimization} prior to \cite{Kobayashi24}; see the
discussion in \cite[\S 4.1]{Kobayashi24} for a brief overview. For the
\eqref{eq:adp} in particular, we simply highlight that the first term
of the first variation formula in \cite[Theorem\ 2.2]{Buttazzo09}
coincides for $p=1$ with the barycentre field expression in
\cref{strongbarycentreapproximation}. However, it seems that
previously, objects like the barycentre field have not been studied
carefully in their own rights, and hence their potential utilities in
analyzing \eqref{eq:objective-minimization} have not been fully
appreciated.

In our case we show that the barycentre field provides a non-obvious
link between two tractable problems; chaining them together ultimately
yields the existence of an atom result, and hence the topological
characterization. We now explain these two problems, which give the
second and third main theorems.

\subsubsection{Nontrivial barycentre field implies existence of an
  atom} \label{sec:nontrivial-bary-existence-of-atom}

Our first step in proving the existence of atoms is to bound the
$\nu$-mass of \textit{noncut points} of an optimizer $\Sigma$ (i.e.\
points $\sigma^* \in \Sigma$ such that $\Sigma \setminus \{\sigma^*\}$
is connected) in terms of the $L^2(\nu)$ norm of the barycentre field.
The intuition is that \cite[Theorem 5.6]{Stepanov04} implies that
noncut points of optimizers must be endpoints, and we expect endpoints
of $\Sigma$ should be atoms for $\nu$ \cite[Discussion following Lemma
7]{Lemenant12}.

\begin{theorem}\label{theorem:introboundingmassofnoncutpoints} (See
  \cref{prop:boundingmassofnoncutpoints} for a more precise
  statement). Let $p \geq 2$ and suppose that $\Sigma \in
  \mathcal{S}_l$ is a solution to the \eqref{eq:adp} and that $\Sigma$
  contains at least two points. Let $\mathcal{B}_{\Sigma}$ be the
  barycentre field of $\Sigma$. Then there exists some constant $C> 0$
  depending only on $l$ and $\mathcal B_{\Sigma}$ such that for all
  noncut points $\sigma^* \in \Sigma$ we have
  \[
    \nu\{\sigma^*\}|\mathcal{B}_{\Sigma}(\sigma^*)| \geq
    C \int_{\Sigma}
    |\mathcal{B}_{\Sigma}(\sigma)|^2 d\nu(\sigma).
  \]
\end{theorem}

\begin{remark}\label{rmk:barycenter-atom}
  Note that every $\Sigma \in \mathcal S_l$ containing at least two
  points has at least two noncut points \cite[\S 47 Theorem
  IV.5]{kuratowski}, whence
  \cref{theorem:introboundingmassofnoncutpoints} shows that
  $L^2(\nu)$-nontriviality of $\mathcal B_\Sigma$ implies existence of
  atoms, and further, that the barycentre field \emph{at} atoms must
  be nontrivial.
\end{remark}

Notice that \cref{theorem:introboundingmassofnoncutpoints} is very
similar to a result first obtained by Buttazzo and Stepanov for $d=2$
\cite{Buttazzo03}*{Proposition 7.1} and later generalized by Stepanov
to $d \geq 2$ \cite[Theorem 5.5]{Stepanov06}. Essentially, they showed
that ``if $\nu$ has an atom $\sigma_0$, then there exists $C>0$ such
that for all noncut points $\sigma^*$, $\nu(\{\sigma^*\}) > C
\nu(\{\sigma_0\})$.'' The difference is that in
\cref{theorem:introboundingmassofnoncutpoints} the lower bound is
given in terms of the barycentre field, which is a subtle---but
powerful---conceptual refinement, as we now explain.

Essentially, \cref{theorem:introboundingmassofnoncutpoints} allows us
to substitute ``$L^2(\nu)$-nontriviality of $\mathcal B_\Sigma$'' in
many places where ``existence of an atom'' is used in the
\eqref{eq:adp} literature to characterize optimizers. A typical
argument in this vein proceeds as follows: One supposes that an
optimizer contains some feature, then uses that feature to show there
exists a perturbation recovering $\varepsilon$ budget with
$o(\varepsilon)$ impact on $\mathscr J_p$, and finally reallocates the
recovered budget to the atom in a way that strictly improves $\mathscr
J_p$, thus contradicting optimality.

For arguments like this, using $\mathcal B_\Sigma$ is more appropriate
than using atoms of $\nu$, since $\mathcal B_\Sigma$ directly encodes
the first-order variational structure of $\mathscr J_p$
(\cref{ifbarycentrenontrivial}). And, in the cases where atoms of
$\nu$ truly \emph{are} the natural objects to study (e.g.\ when
proving optimizers contain only \emph{finitely}-many branches), as
noted in \cref{rmk:barycenter-atom}, $L^2(\nu)$-nontriviality of
$\mathcal B_\Sigma$ recovers the existence of an atom anyways. In this
sense, \cref{theorem:introboundingmassofnoncutpoints} reveals a
fundamental structure in the problem.

To further illustrate the power of
\cref{theorem:introboundingmassofnoncutpoints}, let us note how it
almost trivially solves the topological characterization problem in
the \emph{soft-penalty} version of \eqref{eq:adp}.

\subsubsection{Digression: the soft-penalty problem}\label{sec:soft-penalty}

Let
\[
  \mathcal S = \{\Sigma \subseteq \mathbb R^d \mid \Sigma \text{ is
    compact and connected}\}.
\]
Then for all $\lambda > 0$, the \emph{soft-penalty} problem is to find
solutions to
\begin{align}\label{eqn:soft-lambda}
  \inf_{\Sigma \in \mathcal S} \mathscr J_p^\lambda(\Sigma) \coloneqq
  \inf_{\Sigma \in \mathcal S} (\mathscr J_p(\Sigma) + \lambda
  \mathcal H^1(\Sigma)).
\end{align}
It is straightforward to see that every solution of a soft-penalty
problem \eqref{eqn:soft-lambda} is a solution to a hard-constraint
problem \eqref{eq:adp}: Namely, for fixed $\lambda > 0$, a solution
$\Sigma_\lambda$ to \eqref{eqn:soft-lambda} is a solution to the
hard-constraint problem \eqref{eq:adp} with
$l=\mathcal{H}^1(\Sigma_\lambda)$.

Surprisingly, it is unknown (see \cite[Remark 22]{Lemenant12} and
\cite[\S 1.7.4]{Kobayashi24}) whether the converse holds. We suspect
there are ways that one could prove such equivalence under certain
restrictions on $\mu$; however, discussion of these topics goes beyond
the scope of this paper, and hence we leave it for a future work.

For now, the fact that soft-penalty optimizers are also
hard-constraint optimizers is interesting, because one of our
barycentre field results (\cref{strongbarycentreapproximation}) makes
it almost immediate that soft-penalty optimizers have nontrivial
barycentre fields (\cref{thm:soft-penalty-nontrivial-bary}). Chaining
\cref{thm:soft-penalty-nontrivial-bary} with
\cref{theorem:introboundingmassofnoncutpoints} then directly implies
the topological characterization (\cref{cor:soft-penalty-top-char})
for the soft-penalty problem. The simplicity of this approach stands
in stark contrast with the difficulty of proving the analogy of
\cref{thm:soft-penalty-nontrivial-bary} in the hard-constraint problem
\eqref{eq:adp}; see \cref{sec:bary-nontrivial-summary}.

\begin{proposition} \label{thm:soft-penalty-nontrivial-bary} Let $p
  \geq 1$, taking the extra hypothesis \eqref{eq:regularity-mu} in the
  case $p=1$. Fix $\lambda > 0$ and let $\Sigma_\lambda$ be a
  minimizer of the soft-penalty problem \eqref{eqn:soft-lambda}.
  Suppose that $\mathcal H^1(\Sigma_\lambda) > 0$. Then
  $\Sigma_\lambda$ has nontrivial barycentre field.
\end{proposition}
\begin{proof}
  Suppose, to obtain a contradiction, that $\Sigma_\lambda$ has
  trivial barycentre field. Then by
  \cref{strongbarycentreapproximation}, the perturbation
  $\Sigma_\lambda \mapsto (1-\varepsilon) \Sigma_\lambda$ increases
  $\mathscr J_p$ by at most $o(\varepsilon)$, while on the other hand
  decreasing $\lambda \mathcal H^1$ by $\varepsilon \lambda \mathcal
  H^1(\Sigma_\lambda)$. So, for $\varepsilon$ sufficiently small one
  obtains $\mathscr J_p^\lambda((1-\varepsilon) \Sigma_\lambda) <
  \mathscr J_p^\lambda(\Sigma_\lambda)$, contradicting optimality of
  $\Sigma_\lambda$.
\end{proof}
\begin{corollary} \label{cor:soft-penalty-top-char} Suppose $p \geq 2$
  and let $\lambda > 0$. Let $\Sigma_\lambda$ be a minimizer of the
  soft-penalty problem \eqref{eqn:soft-lambda}. Then $\Sigma_\lambda$
  is a (topological) finite, binary tree, as in
  \cref{main:topologicalcharacterization}.
\end{corollary}
\begin{proof}
  By \cref{thm:soft-penalty-nontrivial-bary}, $\Sigma_\lambda$ has
  nontrivial barycentre field. By virtue of being a soft-penalty
  solution, $\Sigma_\lambda$ also solves the hard-constraint problem
  \eqref{eq:adp} with $l = \mathcal H^1(\Sigma_\lambda)$. Then, since
  $p \geq 2$, \cref{theorem:introboundingmassofnoncutpoints} implies
  $\nu \coloneqq (\pi_{\Sigma_\lambda})_{\#}\mu$ contains an atom.
  Applying \cite{Stepanov06}*{Theorem 5.5} thus yields the topological
  characterization.
\end{proof}

The above discussion demonstrates that
\cref{theorem:introboundingmassofnoncutpoints} is an important step
for understanding the structure of optimizers, especially for the
soft-penalty problem. We conjecture that
\cref{theorem:introboundingmassofnoncutpoints} extends to general $p
\geq 1$, whence \cref{cor:soft-penalty-top-char} would similarly
extend to general $p \geq 1$; this is left for a future work.

Having seen the simplicity of proving the nontriviality result
\cref{thm:soft-penalty-nontrivial-bary} in the soft-penalty
formulation, let us now return to considering the hard-constraint
problem \eqref{eq:adp}, for which the result appears much harder.

\subsubsection{The barycentre field is
  nontrivial for \eqref{eq:adp}} \label{sec:bary-nontrivial-summary}

As mentioned in \cref{rmk:barycenter-atom}, in order to show existence
of atoms using \cref{theorem:introboundingmassofnoncutpoints} one must
prove that the barycentre field is nontrivial. Doing so is easy for
optimizers of the soft-penalty problem \eqref{eqn:soft-lambda}, as we
saw in \cref{thm:soft-penalty-nontrivial-bary}. However, it seems much
more difficult for the hard-constraint problem \eqref{eq:adp}. Our
next main theorem addresses this.
\begin{theorem}\label{theorem:intronontrivialbarycentre} (See also
  \cref{nontrivialbarycentre}). Assume \eqref{eq:regularity-mu} and
  suppose $l > 0$, and let $\Sigma_{\mathrm{opt}}$ be a solution to
  the \eqref{eq:adp}. Assume $p = 2$ or $p > \lbonpval$. Then,
  $\Sigma_{\mathrm{opt}}$ has nontrivial barycentre field.
\end{theorem}

The idea of the proof is as follows: We suppose, to obtain a
contradiction, that $\lVert \mathcal B_\Sigma \rVert_{L^2(\nu)} = 0$.
Then all \emph{continuous} perturbations of $\Sigma$ will impact
$\mathscr J_p$ by at most $o(\varepsilon)$
(\cref{strongbarycentreapproximation}). We exploit this to recover
order $\varepsilon$ budget with only $o(\varepsilon)$ impact to
$\mathscr J_p$ and then reallocate the recovered budget via a
\emph{non-continuous} perturbation to get a competitor $\Sigma^*$.
Careful estimates involving $\mathcal B_\Sigma$ then show that
$\mathscr J_p(\Sigma) - \mathscr J_p(\Sigma^*) > 0$, which contradicts
optimality of $\Sigma$.
\begin{remark}
  We highlight that in the approach just sketched, the supposition
  $\lVert \mathcal B_\Sigma \rVert_{L^2(\nu)} = 0$ aids not only in
  estimating the continuous perturbation, but also in tightening the
  estimates for the \emph{non-continuous} perturbation. Note, \emph{a
    priori} $\mathcal B_\Sigma$ has no compatibility with
  non-continuous perturbations, so these tightened estimates are a bit
  unexpected. But in any case, they play an important role in making
  \cref{theorem:intronontrivialbarycentre} tractable.
\end{remark}

The core ideas of the strategy above were pioneered (though with very
different terminology) by Delattre and Fischer in their work on the
special $p=2$ case of the \emph{length-constrained principal curves}
problem, which is essentially a parametrization-dependent version of
the \eqref{eq:adp}. Using these methods, Delattre and Fischer were
able to establish a result \cite[Lemma\ 3.2]{Delattre17} analogous to
\cref{theorem:intronontrivialbarycentre} for $p =2$; however, as they
specifically noted \cite[\S 1.3]{Delattre17}, they were not able to
show the result for other $p$ values.

We found that by (very roughly speaking) ``factoring'' $\mathscr
J_p(\Sigma) - \mathscr J_p(\Sigma^*)$ into
\begin{enumerate}
  \item a term that looks like the $p=2$ difference $\mathscr
    J_2(\Sigma) - \mathscr J_2(\Sigma^*)$, and
  \item a piece whose precise form (see
    \cref{lem:lower-bound-general}) differs for certain $p$ regimes,
\end{enumerate}
we were able to analyze part (1) using Delattre and Fischer's methods,
while still being able to control part (2) via technical analysis
involving the barycentre field (see \cref{sec:proof-roadmap} for a
detailed overview). Ultimately, this is what allows us to prove
\cref{theorem:intronontrivialbarycentre}, whence
\cref{theorem:introboundingmassofnoncutpoints} gives existence of an
atom, and finally, the topological characterization
\cref{main:topologicalcharacterization} follows via \cite[Theorem\
5.5]{Stepanov06}.

\begin{remark}\label{rmk:strangegap}
  The strange gap in \cref{theorem:intronontrivialbarycentre} between
  $p = 2$ and $p > \frac{1}{2}(3 + \sqrt 5)\approx 2.618 $, arises
  because controlling the ``part (2)'' factor in the regime $2 < p <
  3$ requires controlling a term of the form $\lvert u - v
  \rvert^{p-2}$, which becomes difficult in the range $p \in (2, 3)$.
  In the end, we found that even just treating the $p \in
  (\frac{1}{2}(3 + \sqrt 5), 3)$ regime required quite nontrivial
  technical analysis.
\end{remark}

\begin{remark}\label{rmk:extension-general}
  We do not believe that restriction to $p = 2$ or $p > \frac{1}{2}(3
  + \sqrt{5}) \approx 2.618$ in
  \cref{theorem:intronontrivialbarycentre} is optimal, and in fact
  conjecture that the result should at least hold for $p \geq 2$, and
  likely also $p \geq 1$. An approach, which looks quite nontrivial,
  is to establish the equivalence of the soft-penalty problem
  \eqref{eqn:soft-lambda} and the hard-constraint problem
  \eqref{eq:adp}, and carry the results in \cref{sec:soft-penalty}
  over to the hard-constraint case. We leave this for a future work.
\end{remark}

\subsection{Related works and optimal transport interpretation of the
  average distance problem}
\label{sec:related-works}

Other variants of the problem \eqref{eq:objective-minimization} (i.e.\
using different choices of $\mathscr C$) have applications in many
different contexts, hence the literature contains a range of both
terminology and conceptual frameworks. We may group them into three
loose families: \emph{principal curves/manifolds}
\cites{hastie1984,hastie1989,kegl1998,kegl2000,ozertem2008,gerber2009,biau2011,Kirov16,Lu2016Apr,Delattre17,Lu2020,Warren2025May}
(these works typically view \eqref{eq:objective-minimization} through
a parametrized lens); \emph{the average distance problem}
\eqref{eq:adp}
\cites{Buttazzo02,Buttazzo03,Stepanov04,Stepanov06,Buttazzo09,Lemenant12}
(which takes the parametrization-independent constraint $\mathcal
H^1$); and an unnamed body of recent works more explicitly informed by
optimal transport (OT)
\cites{Chauffert2017Feb,Lebrat2019Apr,Chambolle2023Apr,Kobayashi24}.
More thorough discussion of most of these frameworks can be found in
\cite{Kobayashi24}; here, we just highlight the connection to OT by
briefly recalling the main ideas. (The interested reader is referred
to \cites{Villani09, Santambrogio15, cuturi-peyre2020} for
comprehensive accounts of OT theory).

The goal of typical OT problems is to ``deform'' a given measure
$\mu_0$ into another given measure $\mu_1$ in a maximally-efficient
manner. Perhaps the most famous example is the problem of evaluating
the \emph{Monge-Kantorovich $p$-cost} (also called the
\emph{Wasserstein $p$-cost}), which is defined by
\begin{equation}
  \mathrm{MK}_p^p(\mu_0,\mu_1) = \min_{\gamma \in
    \Gamma(\mu_0,\mu_1)} \int_{\mathbb R^d \times \mathbb R^d}
  \dist^p(x_0,x_1) \, \mathrm{d}\gamma(x_0,x_1), \label{eq:mk-cost}
\end{equation}
where $\Gamma(\mu_0,\mu_1)$ denotes the set of \emph{transport
  plans} $\gamma \in \mathcal P(\mathbb R^d\times \mathbb R^d)$ with
first marginal $\mu_0$ and second marginal $\mu_1$. We may interpret
$\gamma$ as telling us how to ``redistribute'' the mass at each $x_0
\in \supp \mu_0$ so as to yield $\mu_1$, or, symmetrically, how to
``redistribute'' the mass at each $x_1 \in \supp \mu_1$ to yield
$\mu_0$.

As it turns out, the seminal work on the \eqref{eq:adp}
\cite{Buttazzo02} grew out of consideration of a particular OT problem
in which one distinguishes a subset $\Sigma \subseteq \mathbb R^d$
along which transportation of mass is ``free.'' The question of how to
choose the ``best'' compact, connected $\Sigma$ with $\mathcal
H^1(\Sigma) \leq l$ then turned out to be equivalent to the
\eqref{eq:adp}. A similar perspective was given in \cite[Proposition
2.7]{Kobayashi24}, where an elementary argument showed
\[
 \mathscr  J_p(\Sigma) = \inf_{\nu' \in \mathcal P(\Sigma)}
 \mathrm{MK}_p^p(\mu, \nu').
\]
That is, $\mathscr J_p(\Sigma)$ evaluates $\Sigma$ by computing the
best possible transport cost between $\mu$ and a measure supported on
$\Sigma$.

Given these perspectives, the \ref{eq:adp}'s relationship to
\cite{Chambolle2023Apr}, in particular, becomes clearer. There, those
authors investigated a soft-penalty version of
\eqref{eq:objective-minimization} with $\mathscr C(\Sigma) = \mathcal
H^1(\Sigma)$, but replacing $\mathscr J_p(\Sigma)$ with
$\mathrm{MK}_p^p(\mu, \nu_\Sigma)$ where $\nu_\Sigma$ is the uniform
measure on $\Sigma$. We note that in their \S 8 they list a
``topological characterization'' problem analogous to the one we
consider as open in their framework; extending our methods to their
context is an interesting problem which we leave for a future work.

\subsection{Outline}

In \cref{sec:generaltheory}, we recall some general results about the
\eqref{eq:adp}. Some preliminary results are given in
\cref{subsec:preliminaries} and
\cref{sec:restricting-minimizers-to-convex-hull}, such the existence
of minimizers (\cref{prop:existence}) and the fact that average
distance minimizers lie in the convex hull of the support of $\mu$
(\cref{thm:optimizers-in-chull}). In section
\cref{subsec:thebarycentrefield}, we introduce the barycentre field
(\cref{def:barycentrefield}), along with the main results relating it
to the change in the objective value $\mathscr{J}_p(\Sigma)$ under
continuous deformations of $\Sigma$
(\cref{strongbarycentreapproximation}). In
\cref{sec:negligibilityofambiguouslocus} we give our first application
of the barycentre field: Generalizing Delattre and Fischer's $p=2$
result \cite{Delattre17}*{Proposition 3.1} to show that for $p > 1$
(or $p=1$ with an extra hypothesis), any minimizer of the
\eqref{eq:adp} has $\mu$-null ambiguous locus
(\cref{ridgesetnegligable}).

In \cref{sec:topologicalproperties}, we prove the topological
description described in \cref{main:topologicalcharacterization} by
removing the conditional assumption of \cite{Stepanov06}*{Theorem
  5.5}. \cref{subsec:barycentrenontrivialityandthekeylemma} is
dedicated to proving \cref{prop:boundingmassofnoncutpoints}, which
bounds the mass of noncut points in terms of the barycentre field. In
\cref{subsec:minimizershavenontrivialbarycentrefields}, we discuss the
nontriviality of the barycentre field of minimizers
(\cref{nontrivialbarycentre}), which we then combine with the results
of section \cref{subsec:barycentrenontrivialityandthekeylemma} and
\cite{Stepanov06} to prove the topological description in
\cref{sec:topology}.

Finally, in \cref{sec:proof-main}, we provide the proof of
\cref{nontrivialbarycentre}.

\section{General theory}\label{sec:generaltheory}
For the remainder of the paper, we fix
\begin{align*}
  \hbox{ $d \geq 2$, $p \geq 1$, and a compactly supported Borel
  probability measure $\mu$ on $\mathbb{R}^d$.}
\end{align*}
Moreover, we fix some budget $l \geq 0$. Since $\mu$ is fixed for the
remainder of the paper, we suppress the dependence of $\mathscr{J}_p$
on $\mu$ and simply write
\[
  \mathscr{J}_p(\Sigma) \coloneqq \mathscr{J}_p(\mu, \Sigma)
\]
for compact $\Sigma \subseteq \mathbb{R}^d$.

We begin by recalling some preliminary notions and results.
\subsection{Preliminaries}\label{subsec:preliminaries}

First, we recall the existence of minimizers of the \eqref{eq:adp},
e.g. from \cite{Buttazzo02}*{Theorem 2.1}.

\begin{proposition}[Existence of minimizers (see
  \cite{Buttazzo02}*{Theorem 2.1})]\label{prop:existence}
  There exists $\Sigma_{\mathrm{opt}} \in \mathcal{S}_l$ such that
  \[
    \mathscr{J}_p(\Sigma_{\mathrm{opt}}) = \inf_{\Sigma \in
      \mathcal{S}_l}\mathscr{J}_p(\Sigma).
  \]
\end{proposition}

Since we will frequently be referring to the minimum value of
$\mathscr{J}_p$ over $\mathcal{S}_l$, it is helpful to give this
quantity a name:
\begin{equation}\label{eq:jdefinition}
  J(l)\coloneqq \min_{\Sigma \in \mathcal{S}_l}\mathscr{J}_p(\Sigma).
\end{equation}
Next, we establish a basic inequality which we will use throughout the
paper.

\begin{lemma}\label{lemma:basicinequality}
  Fix $p \geq q > 0$. Then, for all $a, b \geq 0$, we have
  \[
    \frac{p}{q} (a^q - b^q)b^{p-q}\leq a^p - b^p \leq \frac{p}{q}(a^q
    - b^q) a^{p-q}.
  \]
  Here we take the convention that $0^0=1$ when $a=0$ or $b=0$.
\end{lemma}

\begin{proof}
  We begin by eliminating some trivial cases. Note that if $a=0$,
  $b=0$, or $a=b$, then the statement holds trivially. So, suppose
  that $a, b > 0$ and $a \ne b$. If $p = q$, the statement again holds
  trivially. Thus, from now on, assume $p > q$.

  Let $f : (0, \infty) \to \mathbb{R}$ be given by $f(x) = x^{p/q}$.
  Notice that $f$ is convex and differentiable, and so for all $x_0,
  x_1$,
  \[
    f'(x_0)(x_1 - x_0) \leq f(x_1) - f(x_0).
  \]
  Writing $f'(x) = \frac{p}{q}x^{(p-q)/q}$ and taking $x_0 = b^q, x_1
  = a^q$ yields
  \[
    \frac{p}{q}b^{p-q}(a^q - b^q) \leq a^p - b^p.
  \]
  Similarly, taking $x_0 = a^q$ and $x_1 = b^q$ yields
  \[
    \frac{p}{q}a^{p-q}(b^q-a^q) \leq b^p - a^p,
  \]
  whence combining these two inequalities yields the desired bound.
\end{proof}

We now recall a basic foundational result from
\cite{Bertsekas78}*{Proposition~7.33}.
\begin{lemma}[Measurable selection]\label{lemma:measurableselection}
  Let $X$ be a metrizable space, $Y$ a compact metrizable space, $D$ a
  closed subset of $X \times Y$, and let $f : D \to \mathbb{R} \cup
  \{- \infty, \infty\}$ be lower semicontinuous. Let $f^*: \pi_X(D)
  \to \mathbb{R}\cup \{- \infty, \infty\}$ be given by
  \[
    f^*(x) = \min_{(x, y) \in \pi_X^{-1}\{x\} \cap D}f(x, y),
  \]
  where $\pi_X: X \times Y \to X, (x, y) \mapsto x$ is the projection.
  Then, $\pi_X(D)$ is closed in $X$, $f^*$ is lower semicontinuous,
  and there exists a Borel-measurable function $\varphi: \pi_X(D) \to
  Y$ such that $\{(x, \varphi(x)) \mid x \in \pi_X(D)\} \subseteq D$
  and $f(x, \varphi(x)) = f^*(x)$ for all $x \in \pi_X(D)$.
\end{lemma}
Our first application of \cref{lemma:measurableselection} is to
establish the existence of a closest-point projection onto $\Sigma$.
\begin{lemma}[Existence of closest-point
  projection]\label{lemma:closestpointprojectionexistence}
  For any compact $\Sigma \subseteq \mathbb{R}^d$ with $\Sigma \ne
  \emptyset$, there exists a Borel measurable $\pi_{\Sigma}:
  \mathbb{R}^d \to \Sigma$ such that
  \begin{align*}
    \hbox{$\dist(x, \Sigma) = | x - \pi_{\Sigma}(x)|$ for all $x \in
    \mathbb{R}^d$.}
  \end{align*}
\end{lemma}
\begin{proof}
  Define the set $D = \{(x, \sigma) \in \mathbb{R}^d \times \Sigma
  \mid \dist(x, \Sigma) = |x - \sigma|\}$, we claim that $D$ is
  closed. Indeed, the map $x \mapsto \dist(x, \Sigma)$ is continuous,
  and thus $F: \mathbb{R}^d \times \Sigma \to \mathbb{R}, (x, \sigma)
  \mapsto \dist(x, \Sigma) - |x - \sigma|$ is continuous, so $D =
  F^{-1}\{0\}$ is closed. Define $f : D \to \mathbb{R}$ by $f(x,
  \sigma) = |x - \sigma|$. Then, $f$ is continuous, and in particular
  lower semicontinuous. So, by measurable selection, there exists a
  Borel-measurable function $\pi_{\Sigma}: \mathbb{R}^d \to \Sigma$
  such that $\dist(x, \Sigma) = |x - \pi_{\Sigma}(x)|$ for all $x \in
  \mathbb{R}^d$.
\end{proof}

\begin{definition}\label{def:closest-proj}
  We will refer to a function satisfying the conclusion of
  \cref{lemma:closestpointprojectionexistence} as a
  \textit{closest-point projection} onto $\Sigma$. Given some compact
  and nonempty $\Sigma \subseteq \mathbb{R}^d$, we define the set of
  closest-point projections onto $\Sigma$ by
  \[
    \Pi_{\Sigma} = \{\pi : \mathbb{R}^d \to \Sigma \mid \pi \text{ is
      measurable and }\dist(x, \Sigma) = |x - \pi(x)| \text{ for all
    }x \in \mathbb{R}^d \}.
  \]
  Then, by \cref{lemma:closestpointprojectionexistence}, we know that
  $\Pi_{\Sigma}$ is nonempty.
\end{definition}

We end this section by defining the ambiguous locus of a set $\Sigma
\subseteq \mathbb{R}^d$.

\begin{definition}[Ambiguous locus]\label{def:ambiguouslocus}
  Let $\Sigma \subseteq \mathbb{R}^d$ be compact and nonempty, and for
  each $x \in \mathbb{R}^d$ consider the set $$\mathcal{P}_{\Sigma}(x)
  = \{\sigma \in \Sigma \mid |x - \sigma| = \dist(x, \Sigma)\}.$$
  Define the \textit{ambiguous locus} of $\Sigma$ to be
  \[
    \mathcal{A}_{\Sigma} = \{x \in \mathbb{R}^d \mid \card(
    \mathcal{P}_{\Sigma}(x)) > 1\},
  \]
  where here $\card( \cdot )$ denotes set cardinality.
\end{definition}
When $\mu(\mathcal{A}_{\Sigma}) = 0$, any two closest-point
projections onto $\Sigma$ are equal $\mu$-almost everywhere.

\subsection{Minimizers may be restricted to the convex hull of
  $\supp \mu$} \label{sec:restricting-minimizers-to-convex-hull}

Since we will frequently make this assumption, we give the following a
label:
\begin{equation}\label{assum:zero-mu}
  \text{for any compact, connected $\Sigma \subseteq \mathbb{R}^d$
    with $\mathcal{H}^1(\Sigma) < \infty$, we have $\mu(\Sigma) = 0$.}
\end{equation}
In particular, this condition is satisfied for any $\mu$
that is absolutely continuous with respect to the $d$-dimensional
Lebesgue measure ($\mu \ll \mathrm{Leb}_d$). Under
\eqref{assum:zero-mu} we may guarantee for all $l \geq 0$ that
optimizers are contained in the convex hull of $\supp \mu$
(\cref{thm:optimizers-in-chull}), which will be used in the proof of
\cref{lem:psi-lowerbound}, an intermediate result used to
establish one of our main theorems, \cref{nontrivialbarycentre}.

Results like \cref{thm:optimizers-in-chull} have been obtained before;
see \cite[Proposition 5.1]{Buttazzo03} for the case $\mu \ll
\mathrm{Leb}_d$, and \cite[Lemma 2.2]{Lu2016Apr} for a proof of the
result in the \emph{penalized principal curves} problem (roughly, a
soft-penalty, parametrization-dependent version of the
\eqref{eq:adp}).

We thank an anonymous reviewer for suggesting the following approach,
which allowed us to greatly simplify our original proof.

\begin{lemma} \label{prop:h1-projection-contraction} Let $C \subseteq
  \mathbb R^d$ be compact and connected, and suppose $0 < \mathcal
  H^1(C) < \infty$. Let $P$ be a closed half-space such that $C \cap
  P$, $C \cap (\mathbb R^d \setminus P) \neq \varnothing$. Then
  $\mathcal H^1(\pi_P(C)) < \mathcal H^1(C)$.
\end{lemma}

\begin{proof}[Sketch]
  Fix $s_{0} \in C \cap (\mathbb R^d \setminus P)$ and a
  minimal-length arc $A \subseteq C$ connecting $s_0$ to $\partial P$.
  By \cite[Theorem\ 4.4.7]{ambrosio-tilli} there exists an injective,
  1-Lipschitz map $\gamma$ parametrizing $A$, and Rademacher's theorem
  gives
  \[
    \gamma(\mathcal H^1(A)) - \gamma(0) = \int_0^{\mathcal H^1(A)}
    \dot \gamma \ \mathrm{d}t.
  \]
  Resolving $\dot \gamma$ into parallel/perpendicular components to
  $\partial P$, $d(s_{0}, P) > 0$ implies the orthogonal component is
  nontrivial, whence $\mathcal H^1(A) > \mathcal H^1(\pi_P(A))$.
\end{proof}

\begin{lemma}[Minimizers lie in the convex hull of
  $\mu$]\label{thm:optimizers-in-chull} Assume
  \eqref{assum:zero-mu}. Let $\Sigma \in \mathcal{S}_l$ be a
  minimizer. Then $\Sigma \subseteq S \coloneqq
  \mathrm{ConvexHull}(\supp \mu)$.
\end{lemma}
\begin{proof}
  Suppose, to obtain a contradiction, that $\Sigma \not \subseteq S$.
  Since $S$ is equal to the intersection of all closed half-spaces
  containing $\supp \mu$, we may find one such half-space $P$ such
  that $\Sigma \setminus P \ne \emptyset$. By
  \cref{prop:h1-projection-contraction},
  \[
    \mathcal{H}^1(\pi_P(\Sigma \setminus P)) < \mathcal{H}^1(\Sigma
    \setminus P),
  \]
  so $\mathcal{H}^1(\pi_{P}(\Sigma)) < l$. For concision denote
  $\Sigma_P \coloneqq \pi_P(\Sigma)$ and $\delta \coloneqq l -
  \mathcal H^1(\pi_{P}(\Sigma))$.

  Since $\supp \mu \subseteq P$, we have
  \begin{equation}
    \dist(x, \pi_{\Sigma_P}(x)) \leq \dist(x,
    \pi_{P}(\pi_{\Sigma}(x))) \leq \dist(x,
    \pi_{\Sigma}(x)); \label{eq:projection-non-decreasing}
  \end{equation}
  $\mu$-a.e., it follows that \(\mathscr{J}_p(\pi_{P}(\Sigma)) \leq
  \mathscr{J}_p(\Sigma).\) Next, by \eqref{assum:zero-mu}, $\supp \mu
  \setminus \Sigma_P \neq \emptyset$; hence select some $x_0 \in \supp
  \mu \setminus \Sigma_P$, and let $\sigma_0 \in \mathcal
  P_{\Sigma_P}(x_0)$ (see \cref{def:ambiguouslocus}). Define the
  competitor $\Sigma^* = \Sigma_P \cup [\sigma_0,\ \sigma_0 +
  \delta(x_0 - \sigma_0)]$. One may show that for $r > 0$ sufficiently
  small, for all $x' \in B_r(x_0)$,
  \eqref{eq:projection-non-decreasing} becomes
  \[
    \dist(x', \pi_{\Sigma^*}(x')) < \dist(x', \pi_{\Sigma_P}(x_0))
    \leq \dist(x', \pi_\Sigma(x')),
  \]
  whence $\mathscr J_p(\Sigma^*) < \mathscr J_p(\Sigma)$. But note,
  $\mathcal H^1(\Sigma^*) \leq \mathcal H^1(\Sigma_P) + \delta = l$,
  so $\Sigma^* \in \mathcal S_l$, contradicting optimality of
  $\Sigma$. So, $\Sigma \subseteq S$, as desired.
\end{proof}

\subsection{The barycentre field}\label{subsec:thebarycentrefield}

Now, we are ready to define the barycentre field. The following
definition appeared in \cite{Kobayashi24}.
\begin{definition}[Barycentre field]\label{def:barycentrefield}
  Let $\Sigma \subseteq \mathbb{R}^d$ be compact and nonempty, and let
  $\pi_{\Sigma} \in \Pi_{\Sigma}$ be a closest-point projection onto
  $\Sigma$. Let $$\nu_{\pi_{\Sigma}} = (\pi_{\Sigma})_{\#}\mu$$ be the
  pushforward measure on $\Sigma$. Let $(\{\rho_{\sigma}\}_{\sigma \in
    \Sigma}, \nu_{\pi_{\Sigma}})$ be the disintegration of $\mu$ by
  $\pi_{\Sigma}$. Then, we define the \textit{barycentre field of}
  $\pi_{\Sigma}$, a vector field along $\Sigma$, by
  \[
    \mathcal{B}_{\pi_{\Sigma}}(\sigma) \coloneqq p
    \int_{\pi_{\Sigma}^{-1}\{\sigma\}}|x - \pi_{\Sigma}(x)|^{p-2}(x -
    \pi_{\Sigma}(x)) d\rho_{\sigma}(x), \quad \hbox{for $\sigma \in
      \Sigma.$}
  \]
\end{definition}
\begin{remark} \label{rem:bad-regularity}
  One may view $\mathcal B_{\pi_\Sigma}$ as a function
  $\mathcal{B}_{\pi_{\Sigma}} : \Sigma \to \mathbb{R}^d$, though it is
  often helpful to instead picture a vector with tail $\sigma$ and
  head $\sigma + \mathcal B_{\pi_\Sigma}(\sigma)$. In any case, as
  seen in \cite[CE.\ 4.8]{Kobayashi24}, $\mathcal B_{\pi_\Sigma}$ can
  badly lack regularity, sometimes being discontinuous even when
  $\Sigma$ is a $C^1$ manifold and $\mu$ is uniform. Thus, an
  approximation result (\cref{approximatebysmooth}) is necessary to
  ensure we can ``follow'' it and remain in $\mathcal S_l$.
\end{remark}

As mentioned previously, the barycentre field essentially encodes the
``gradient'' of the objective functional $\mathscr{J}_p$. This feature
is made precise by the following result, which is an adaptation of
\cite{Kobayashi24}*{Theorem 4.8 and Corollary 4.10} to the
non-parameterized context. In fact, \cite{Kobayashi24}*{Theorem 4.8}
already implies a stronger result than
\cref{strongbarycentreapproximation}, namely that equality holds in
\cref{eq:fbary-grad-J-min} (this justifies us calling the barycentre
field a ``gradient''). However, the inequality
\eqref{eq:fbary-grad-J-min} is all that is needed for the purposes of
this paper, hence we give a slightly-simpler, self-contained proof of
this fact.

\begin{proposition}[The ``gradient'' of $\mathscr{J}_p$]
  \label{strongbarycentreapproximation}
  Let $\Sigma \in \mathcal{S}_l$ and let $\xi : \Sigma \to
  \mathbb{R}^d$ be continuous. For $\epsilon > 0$, define
  \begin{equation}\label{eq:perturbedsigma}
    \Sigma_{\epsilon, \xi} \coloneqq \{\sigma + \epsilon \xi(\sigma) \
    | \ \sigma \in \Sigma\} = (\mathbf{1} + \epsilon \xi)(\Sigma).
  \end{equation}
  Suppose $p > 1$ or that $p = 1$ and $\mu(\Sigma) = 0$.
  Then
  \begin{equation} \lim_{\varepsilon \to 0^{+}}
    \frac{\mathscr{J}_p(\Sigma_{\varepsilon, \xi}) -
      \mathscr{J}_p(\Sigma)}{\varepsilon} \leq \max_{\pi_{\Sigma}\in
      \Pi_{\Sigma}}- \int_\Sigma \xi(\sigma) \cdot
    \mathcal{B}_{\pi_\Sigma}(\sigma) \,
    d\nu_{\pi_{\Sigma}}(\sigma). \label{eq:fbary-grad-J-min}
  \end{equation}
\end{proposition}
\begin{proof}
  Fix $\pi_{\Sigma}\in \Pi_{\Sigma}$. For concision, for all $x$, let
  \[
    \sigma_x \coloneqq \pi_\Sigma(x) \qquad \text{and} \qquad
    \sigma_x^{\varepsilon} \coloneqq \pi_{\Sigma}(x) + \epsilon
    \xi(\pi_{\Sigma}(x)).
  \]
  Notice that for all $x$ we have $\sigma_x^\varepsilon \in
  \Sigma_{\epsilon,\xi}$, so $\lvert x - \pi_{\Sigma_{\varepsilon,
      \xi}}(x) \rvert \leq \lvert x - \sigma_x^\varepsilon \rvert$.
  Consequently,
  \[
    \mathscr{J}_p(\Sigma_{\epsilon, \xi}) - \mathscr{J}_p(\Sigma) \leq
    \int_{\mathbb{R}^d} |x - \sigma_x^\epsilon|^p - |x - \sigma_x|^p
    d\mu(x).
  \]
  Now, taking $a = |x - \sigma_x^\epsilon|$, $b = |x - \sigma_x|$, and
  $q = 1$ in \cref{lemma:basicinequality}, we have
  \[
    |x - \sigma_x^\epsilon|^p - |x - \sigma_x|^{p} \leq p\big(|x -
    \sigma_x^\epsilon| - |x - \sigma_x|\big)|x - \sigma_x|^{p-1}.
  \]

  We claim that for $\mu$-a.e.\ $x \in \mathbb{R}^d$ we have
  \begin{align}
    \lim_{\epsilon \to 0^{+}}\frac{p\big(|x - \sigma_x^\epsilon| - |x
    - \sigma_x|\big)|x - \sigma_x|^{p-1}}{\varepsilon}
    &= - p\xi(\sigma_x)\cdot (x - \sigma_x)|x -
      \sigma_x|^{p-2}, \label{eq:limepsilonto0equalsbarycentreinnterterm}
  \end{align}
  taking the convention that the right hand side is zero when
  $(x-\sigma_x) = 0$ and $p > 1$. To that end we proceed by casework.

  \emph{Case 1:} Suppose $x \notin \Sigma$. Then $\lvert x - \sigma_x
  \rvert \neq 0$, whence $\frac{1}{\lvert x - \sigma_x \rvert}$ is
  well-defined. So,
  \begin{align*}
    |x - \sigma_x^\epsilon|
    &= \Big[(x - \sigma_x - \epsilon \xi(\sigma_x)) \cdot (x -
      \sigma_x - \epsilon \xi(\sigma_x))\Big]^{1/2} \\
    &= \Big[\lvert x - \sigma_x \rvert^2 - 2\varepsilon \xi(\sigma_x)
      \cdot (x - \sigma_x) + \varepsilon^2 \lvert \xi(\sigma_x)
      \rvert^2 \Big]^{1/2}\\
    &=\lvert x - \sigma_x \rvert \bigg[1 - 2\epsilon
      \xi(\sigma_x)\cdot \frac{(x - \sigma_x)}{|x - \sigma_x|} +
      \epsilon^2\frac{|\xi(\sigma_x)|^2}{|x -
      \sigma_x|^2}\bigg]^{1/2}.
  \end{align*}
  Applying the Binomial approximation then yields
  \[
    \lvert x - \sigma_x^\epsilon \rvert = |x - \sigma_x| - \epsilon
    \xi(\sigma_x)\cdot (x - \sigma_x) + O(\epsilon^2).
  \]
  Thus, \eqref{eq:limepsilonto0equalsbarycentreinnterterm} holds when
  $x \notin \Sigma$. It remains to show the result when $x \in
  \Sigma$.

  \emph{Case 2:} Suppose $x \in \Sigma$. Then we have two trivial
  subcases.

  \emph{Subcase 2.i:} First, suppose $p > 1$. Then $|x - \sigma_x| =
  0$, and so the left side of
  \eqref{eq:limepsilonto0equalsbarycentreinnterterm} is $0$, whence by
  our choice of convention for the right hand side we get the equality
  in \eqref{eq:limepsilonto0equalsbarycentreinnterterm} as desired.

  \emph{Subcase 2.ii:} Next, suppose $p =1$. Recall that in this case
  we assume the additional hypothesis $\mu(\Sigma) = 0$. Thus
  $\mu$-a.e.\ $x \not \in \Sigma$, and so the reasoning of Case 1
  shows that \eqref{eq:limepsilonto0equalsbarycentreinnterterm} holds
  $\mu$-a.e.\ and there is nothing to prove.

  \medskip

  Since Cases 1 and 2 are exhaustive,
  \eqref{eq:limepsilonto0equalsbarycentreinnterterm} holds $\mu$-a.e.\
  as claimed.

  \medskip

  Now we will conclude by applying dominated convergence with
  \eqref{eq:limepsilonto0equalsbarycentreinnterterm} to obtain
  \eqref{eq:fbary-grad-J-min}. For this we must first establish a
  uniform bound on the terms inside the limit on the left hand side of
  \eqref{eq:limepsilonto0equalsbarycentreinnterterm}. To that end,
  observe that for any $\varepsilon > 0$ the the reverse triangle
  inequality gives $\big\lvert \lvert x - \sigma_x^\varepsilon \rvert
  - \lvert x - \sigma_x \rvert \big\rvert \leq \lvert \varepsilon
  \xi(\sigma_x) \rvert$, so
  \begin{align*}
    \bigg\lvert \frac{p \big(|x - \sigma_x^\epsilon| - |x -
    \sigma_x|\big)}{\varepsilon} |x - \sigma_x|^{p-1}\bigg\rvert
    &\leq p|\xi(\sigma_x)||x - \sigma_x|^{p-1} \\
    &\leq p (\sup \lvert \xi \rvert) \diam(\supp \mu)^{p-1}.
  \end{align*}
  Recalling that $\xi$ is continuous, $\Sigma$ is compact, and $\mu$
  is compactly-supported, we see the right hand side is finite, thus
  giving the desired uniform bound. Applying the dominated convergence
  theorem and the $\mu$-a.e.\ inequality
  \eqref{eq:limepsilonto0equalsbarycentreinnterterm} thus gives
  \begin{align*}
    \lim_{\epsilon \to 0^+}\frac{\mathscr{J}_p(\Sigma_{\epsilon,
    \xi}) - \mathscr{J}_p(\Sigma)}{\epsilon}
    &\leq \int_{\mathbb{R}^d} \lim_{\epsilon \to
      0^+} \frac{p\big(|x - \sigma_x^\epsilon| - |x - \sigma_x|\big)
      |x - \sigma_x|^{p-1}}{\varepsilon} d\mu(x) \\
    &= - \int_{\mathbb{R}^d}p \xi(\sigma_x)\cdot(x - \sigma_x)|x -
      \sigma_x|^{p-2}d\mu(x) \\
    &= - \int_{\Sigma}\xi(\sigma)\cdot
      \mathcal{B}_{\pi_{\Sigma}}(\sigma)d\nu_{\pi_{\Sigma}}(\sigma),
  \end{align*}
  as was to be shown.
\end{proof}

\begin{remark} \label{rem:error-estimates} Similarly to the above, the
  idea of the proof of the equality case in
  \cite{Kobayashi24}*{Theorem 4.8} is to get lower/upper bounds for
  the first variation (left side of \eqref{eq:fbary-grad-J-min}) by
  expanding certain expressions of the form $|u - \epsilon v|^p$ to
  first order in $\epsilon$. In the limit as $\epsilon \to 0$ these
  first-order terms converge to the right side of
  \eqref{eq:fbary-grad-J-min} while the higher-order terms decay
  rapidly, thus yielding the equality.

  The takeaway is that, at least when it comes to regularity,
  $\mathscr J_p$ inherits the flavour of $x \mapsto |x|^p$. Indeed,
  the special $p=1$ hypothesis $\mu(\Sigma) = 0$ is
  essentially used to handle the nondifferentiability of $|u -
  \epsilon v|$ at $u=0$. When $p > 2$ (or when $p=2$ with an extra
  hypothesis), similar ideas may be used to derive a formula for the
  second variation, whence for small $\epsilon$ we obtain the
  expansion $\mathscr J_p(\Sigma_{\epsilon, \xi}) =
  \mathscr{J}_p(\Sigma) + \epsilon \cdot (\text{RHS of
    \eqref{eq:fbary-grad-J-min}}) + O(\epsilon^2)$. Note, the fact
  that the error term in this expansion is $O(\epsilon^2)$ (instead of
  the coarser $o(\epsilon)$) will inform our intuition for the proof
  of the topological description result in \cref{sec:proof-roadmap}.
\end{remark}

Now that we have established the role the barycentre field plays as
the ``gradient'' of $\mathscr{J}_p$, we provide a definition which
relates to whether or not $\Sigma$ is a critical point of
$\mathscr{J}_p$ under continuous perturbations.
\begin{definition}\label{def:trivialbarycentrefield}
  Let $\pi_{\Sigma} \in \Pi_{\Sigma}$, and let $\nu_{\pi_{\Sigma}}=
  (\pi_{\Sigma})_{\#}\mu$. We say $\pi_{\Sigma}$ has \textit{trivial}
  barycentre field if
  \[
    \nu_{\pi_{\Sigma}}(\{\sigma \in \Sigma \mid
    \mathcal{B}_{\pi_{\Sigma}}(\sigma) \ne 0\}) = 0.
  \]
  Otherwise, we say the barycentre field of $\pi_{\Sigma}$ is
  \textit{nontrivial}.
\end{definition}
We expect that any minimizer $\Sigma$ of the \eqref{eq:adp} should
have a nontrivial barycentre field, for any choice of closest point
projection. The reasoning for this intuition will become much clearer
in \cref{sec:topologicalproperties}, where we will discuss the
relationship between the barycentre field and atoms of $\nu$; that is,
points $\sigma \in \Sigma$ such that $\nu(\{\sigma\}) > 0$.

One of the uses of the barycentre field in studying the \eqref{eq:adp}
problem is in its ability to produce upper bounds on the quantity
\[
  \lim_{\epsilon \to 0^+} \frac{J(l + \epsilon) - J(l)}{\epsilon}
  \quad \hbox{for $J (l)$ given as in \eqref{eq:jdefinition}.}
\]
The idea behind producing these bounds is as follows. Suppose we are
given an optimizer $\Sigma \in \mathcal{S}_l$ with nontrivial
barycentre field $\mathcal{B}_{\pi_{\Sigma}}$. We then find a
$\frac{1}{l}$-Lipschitz map $\Xi: \mathbb{R}^d \to \mathbb{R}^d$ such
that $\xi \coloneqq \Xi|_{\Sigma}$ approximates
$\mathcal{B}_{\pi_{\Sigma}}$ in $L^2(\Sigma, \nu)$. Then
$\Sigma_{\epsilon, \xi} \in \mathcal{S}_{l + \epsilon}$, so
\eqref{eq:fbary-grad-J-min} allows us to bound $\lim_{\epsilon \to
  0^+} \frac{J(l + \epsilon) - J(l)}{\epsilon}$ from above by
$-\alpha||\mathcal{B}_{\pi_{\Sigma}}||^2_{L^2(\Sigma, \nu)}$ for some
constant $\alpha > 0$ depending on $\xi$. We will make this idea
precise in \cref{approximatebysmooth}, but first, let us start with a
density result.

\begin{lemma}\label{lemma:lipschitzdense}
  Fix any finite measure $\nu$ over $\Sigma$, and endow $L^2(\Sigma,
  \nu; \mathbb{R}^d)$ with the standard topology induced by $\lVert
  \cdot \rVert_{L^2(\Sigma, \nu; \mathbb{R}^d)}$. Then the set
  $\mathcal{L}(\Sigma) \coloneqq \{\Xi|_{\Sigma} \mid \Xi:
  \mathbb{R}^d \to \mathbb{R}^d \text{ is Lipschitz} \}$ is dense in
  $L^2(\Sigma, \nu; \mathbb{R}^d)$.
\end{lemma}
\begin{proof}
  By analyzing the component functions individually, it suffices to
  just prove the $d=1$ case, whence we suppress writing the $\mathbb
  R^d$'s. To that end, one may first show $C(\Sigma)$ is dense in
  $(L^2(\Sigma, \nu), \lVert \cdot \rVert_{L^2(\Sigma, \nu)})$ by
  fixing $f \in L^2(\Sigma, \nu)$, applying a textbook result like
  \cite[Proposition 6.7]{folland}, and then mollifying to yield a
  $C(\Sigma)$ approximant. Thus it suffices to show
  $\mathcal{L}(\Sigma)$ is dense in $(C(\Sigma), \lVert \cdot
  \rVert_{L^2(\Sigma,\nu)})$. Observe that
  \[
    ||f - g||_{L^2(\Sigma, \nu)} \leq \nu(\Sigma)\sup_{\sigma\in
      \Sigma}|f(\sigma) - g(\sigma)|.
  \]
  By the Stone-Weierstrass theorem $\mathcal{L}(\Sigma)$ is dense in
  $C(\Sigma)$ with respect to the uniform topology, and applying this
  to the above inequality yields the claim.
\end{proof}
We use this lemma to prove:
\begin{proposition}[Approximation of
  $\mathcal{B}_{\pi_{\Sigma}}$]\label{approximatebysmooth}
  Let $p \geq 1$. Suppose that $\pi_{\Sigma} \in \Pi_\Sigma$ has
  nontrivial barycentre field. Then, there exists a Lipschitz map $\xi
  : \mathbb{R}^d \to \mathbb{R}^d$ such that
  \[
    \int_{\Sigma} \xi(\sigma) \cdot
    \mathcal{B}_{\pi_{\Sigma}}(\sigma) d \nu_{\pi_{\Sigma}}(\sigma) >
    \frac{1}{2}\int_{\Sigma}|\mathcal{B}_{\pi_{\Sigma}}(\sigma) |^2 d
    \nu_{\pi_{\Sigma}}(\sigma) > 0.
  \]
\end{proposition}
\begin{proof}
  First, observe that the right inequality follows immediately from
  nontriviality of the barycentre field. For the remainder: Denote
  $\nu = \nu_{\pi_{\Sigma}}\coloneqq (\pi_{\Sigma})_{\#}\mu$, the
  $L^2(\Sigma, \nu; \mathbb{R}^d)$ inner product by $\langle f,g
  \rangle \coloneqq \int_\Sigma f \cdot g\ d\nu$, and the associated
  $L^2(\Sigma, \nu; \mathbb{R}^d)$ norm by $\lVert \cdot \rVert$. Then
  it suffices to find a Lipschitz $\xi$ with $\langle \xi, \mathcal
  B_{\pi_\Sigma} \rangle > \frac{1}{2} \lVert \mathcal B_{\pi_\Sigma}
  \rVert^2$.

  Observe that we have the uniform bound
  $|\mathcal B_{\pi_\Sigma} (\sigma)| \leq \diam(\supp
  \mu)^{p-1}$. Since $|\nu| < \infty$ this gives $\mathcal
  B_{\pi_\Sigma} \in L^2(\nu)$; in particular $\lVert \mathcal
  B_{\pi_\Sigma} \rVert < \infty$. So, pick some $0 < \delta
  < \frac{1}{2}  \lVert \mathcal B_{\pi_\Sigma} \rVert$. By
  \cref{lemma:lipschitzdense}, there exists a Lipschitz $\xi :
  \mathbb{R}^d \to \mathbb{R}^d$ such that $ \lVert \xi - \mathcal
  B_{\pi_\Sigma} \rVert < \delta$, and thus
  \begin{align*}
    \langle \xi, \mathcal B_{\pi_\Sigma} \rangle
    & = \lVert \mathcal B_{\pi_\Sigma} \rVert^2 -
      \langle \mathcal B_{\pi_\Sigma} - \xi,\ \mathcal
      B_{\pi_\Sigma} \rangle \\
    & \geq \lVert \mathcal B_{\pi_\Sigma} \rVert^2 -
      \lVert \mathcal B_{\pi_\Sigma} - \xi
      \rVert \lVert \mathcal B_{\pi_\Sigma}
      \rVert \\
    & > \lVert \mathcal B_{\pi_\Sigma} \rVert^2 - \delta \lVert
      \mathcal B_{\pi_\Sigma} \rVert \\
    & > \frac{1}{2} \lVert \mathcal B_{\pi_\Sigma} \rVert^2,
  \end{align*}
  as required.
\end{proof}

\begin{corollary}[Nontrivial barycentre field implies strictly
  decreasing optimal value $J$]
  \label{ifbarycentrenontrivial}
  Suppose $p > 1$, or $p =1$ and $\mu(\Sigma) = 0$ for all $\Sigma \in
  \mathcal{S}_l$. Let $\Sigma \in \mathcal{S}_l$, and suppose that
  $\pi_{\Sigma}$ has nontrivial barycentre field. Then, for all
  sufficiently small $\epsilon$, there exists some $\Sigma' \in
  \mathcal{S}_{l + \epsilon}$ such that
  \[
    \mathscr{J}_p(\Sigma') - \mathscr{J}_p(\Sigma) < - C \epsilon,
  \]
  for some constant $C > 0$. In particular, we have that
  \[
    \lim_{\epsilon \to 0^+} \frac{J(l + \epsilon) - J(l)}{\epsilon}
    \leq - C < 0.
  \]
\end{corollary}

\begin{proof}
  Since $\pi_{\Sigma}$ has nontrivial barycentre field, by
  \cref{approximatebysmooth} there exists a Lipschitz map $\xi :
  \mathbb{R}^d \to \mathbb{R}^d$ such that
  \[
    \int_{\Sigma} \xi(\sigma) \cdot \mathcal{B}_{\pi_{\Sigma}}(\sigma)
    d\nu_{\pi_{\Sigma}} (\sigma) >
    \frac{1}{2}\int_{\Sigma}|\mathcal{B}_{\pi_{\Sigma}}(\sigma)|^2
    d\nu_{\pi_{\Sigma}} (\sigma) > 0.
  \]
  Let $\eta = (l \, \mathrm{Lip}(\xi))^{-1}$, then the map $\mathbf{1}
  + \epsilon \eta \xi$, that is, $\sigma \mapsto \sigma + \epsilon
  \eta \xi(\sigma)$, has Lipschitz constant $\mathrm{Lip}(\mathbf{1} +
  \epsilon \eta \xi) \le (1+\epsilon/l)$, so $\Sigma_{\epsilon,
    \eta\xi} = \{\sigma + \epsilon \eta \xi(\sigma) \mid \sigma \in
  \Sigma\}$ satisfies
  \[
    \mathcal{H}^1(\Sigma_{\epsilon, \eta \xi}) \le (1+\epsilon/l)
    \mathcal{H}^1 (\Sigma) \leq l + \epsilon.
  \]
  In particular, for any $\epsilon > 0$ we have $\Sigma_{\epsilon,
    \eta \xi} \in \mathcal{S}_{l + \epsilon}$. So, letting
  \[
    C =
    \frac{\eta}{2}\int_{\Sigma}|\mathcal{B}_{\pi_{\Sigma}}(\sigma)|^2
    d\nu_{\pi_{\Sigma}} (\sigma) > 0,
  \]
  and applying \cref{strongbarycentreapproximation}, for all
  sufficiently small $\epsilon > 0$ we have
  \[
    \mathscr{J}_p(\Sigma_{\epsilon, \eta \xi}) - \mathscr{J}_p(\Sigma)
    < -C \epsilon.
  \]
  Taking $\Sigma' = \Sigma_{\epsilon, \eta \xi}$ yields the
  desired result.
\end{proof}

\subsection{Negligibility of the ambiguous
  locus}\label{sec:negligibilityofambiguouslocus}

To conclude this section, we will use the barycentre field to
generalize \cite{Delattre17}*{Proposition 3.1}, which proves for $p=2$
that the ambiguous locus of any minimizer of the related
\textit{length-constrained principal curves problem} has $\mu$-measure
zero.

We first prove in the next lemma that if the net barycentre field, a
single vector in $\mathbb R^d$, is nonzero, then shifting the whole
set along it is an effective way to strictly decrease the objective
function $\mathscr J_p$. Since the set $\mathcal{S}_l$ is closed under
translation, this shows that the net barycentre field of any minimizer
$\Sigma \in \mathcal{S}_l$ of the \eqref{eq:adp} is zero.

\begin{lemma}[{The objective decreases along the net
    barycentre field direction}]
  \label{pmeanlemma}
  Assume $p > 1$, or $p=1$ with the extra hypothesis $\mu(\Sigma)
  = 0$. Suppose $\Sigma \subseteq \mathbb{R}^d$ is compact and
  nonempty, and let $\pi_{\Sigma}\in \Pi_{\Sigma}$. Let
  \[
    \mathcal{B}^{\mathrm{net}}_{\pi_{\Sigma}} =
    \int_{\Sigma}\mathcal{B}_{\pi_{\Sigma}}(\sigma)
    d\nu_{\pi_{\Sigma}}(\sigma),
  \]
  and define $\Sigma_{\epsilon} = \Sigma + \epsilon
  \mathcal{B}^{\mathrm{net}}_{\pi_{\Sigma}}$. Then,
  \[
    \lim_{\epsilon \to 0} \frac{\mathscr{J}_p(\Sigma_{\epsilon})
      -\mathscr{J}_p(\Sigma)}{\epsilon} \leq
    -|\mathcal{B}^{\mathrm{net}}_{\pi_{\Sigma}}|^2.
  \]
  In particular, for any minimizer $\Sigma \in \mathcal{S}_l$ of the
  \eqref{eq:adp}, we have $\mathcal{B}^{\mathrm{net}}_{\pi_{\Sigma}} =
  0$.
\end{lemma}

\begin{proof}
  Applying \cref{strongbarycentreapproximation} with $\xi(\sigma) =
  \mathcal{B}^{\mathrm{net}}_{\pi_{\Sigma}}$ immediately yields
  \[
    \lim_{\epsilon\to 0} \frac{\mathscr{J}_p(\Sigma_{\epsilon}) -
      \mathscr{J}_p(\Sigma)}{\epsilon} \leq - \int_{\Sigma}
    \mathcal{B}^{\mathrm{net}}_{\pi_{\Sigma}} \cdot
    \mathcal{B}_{\pi_{\Sigma}}(\sigma)d\nu_{\pi_{\Sigma}}(\sigma) = -
    |\mathcal{B}^{\mathrm{net}}_{\pi_{\Sigma}}|^2,
  \]
  as desired.
\end{proof}
\begin{remark}\label{rm:net-bary}
  Observe from the definition of the net barycentre field that
  \begin{align}\label{eq:net-bary}
    \mathcal{B}^{\mathrm{net}}_{\pi_{\Sigma}} =
    p \int_{\mathbb{R}^d}|x - \pi_{\Sigma}(x)|^{p-2}(x -
    \pi_{\Sigma}(x)) d \mu(x).
  \end{align}
  Suppose $\Sigma \in \mathcal S_l$ is a minimizer of the
  \eqref{eq:adp}. Then by \cref{pmeanlemma} we have
  $\mathcal{B}_{\pi_{\Sigma}}^{\mathrm{net}} = 0$, so rearranging
  \eqref{eq:net-bary} gives
  \begin{equation}
    \int_{\mathbb R^d} \lvert x - \pi_\Sigma(x) \rvert^{p-2} x d\mu(x)
    = \int_{\mathbb R^d} \lvert x-\pi_\Sigma(x) \rvert^{p-2}
    \pi_\Sigma(x) d\mu(x). \label{eq:net-bary-optimizer-invariant}
  \end{equation}
  Consider the special case $p=2$. Examination of the integrand in
  \eqref{eq:net-bary} shows when $x \in \Sigma$, taking the convention
  $\lvert x - \pi_\Sigma(x) \rvert^{p-2} = 0^0 = 1$ causes no
  problems, so \eqref{eq:net-bary-optimizer-invariant} reduces to
  $\mathbb{E}_\mu[x] = \mathbb{E}_{\mu}[\pi_\Sigma(x)]$. This
  invariant appeared in \cite{Delattre17}*{Remark 2}, and plays a key
  role in their result \cite{Delattre17}*{Proposition 3.1}.
\end{remark}

Accordingly, by replacing the invariant \cite{Delattre17}*{Remark 2}
with the conclusion of \cref{pmeanlemma}, we are able to generalize
\cite{Delattre17}*{Proposition 3.1} to the case of $p > 1$ (or $p =1$
with an extra hypothesis).

\begin{proposition}[Negligibility of the ambiguous
  locus]\label{ridgesetnegligable}
  Suppose $p > 1$, or $p=1$ with the extra hypothesis $\mu(\Sigma) =
  0$, and let $\Sigma \in \mathcal{S}_l$ be a minimizer of the
  \eqref{eq:adp}. Then $\mu(\mathcal{A}_{\Sigma}) = 0$.
\end{proposition}

\begin{proof}
  We first show that for each $j \in \{1, \dots, d\}$, there exist
  well-defined, Borel-measurable functions $P^{(j)}_{0}, P^{(j)}_{1}
  \in \Pi_\Sigma$ (see \cref{def:closest-proj}) such that defining
  $\pi_j$ via $(x^1, \ldots, x^d) \xmapsto{\pi_j} x^j$ and $\mathcal
  P_\Sigma(x)$ as in \cref{def:ambiguouslocus} we have
  \begin{equation*}
    \pi_j \circ P^{(j)}_{0}(x) = \min
    \pi_j(\mathcal{P}_{\Sigma}(x))
    \qquad \text{and} \qquad
    \pi_j \circ P^{(j)}_{1}(x) = \max
    \pi_j(\mathcal{P}_{\Sigma}(x)).
  \end{equation*}

  Fix a $j \in \{1, \dots, d\}$; as in
  \cref{lemma:closestpointprojectionexistence} let $D = \{(x, \sigma)
  \in \mathbb{R}^d\times \Sigma \mid \sigma \in
  \mathcal{P}_{\Sigma}(x)\}$, and note $D$ is closed. Define $f : D
  \to \mathbb{R}$ by $ f(x, \sigma) = -\pi_j(\sigma)$; then $f$ is
  continuous, and so \cref{lemma:measurableselection} yields a
  Borel-measurable $P^{(j)}_1 : \mathbb{R}^d \to \Sigma$ such that for
  each $x \in \mathbb R^d$ we have $(x, P^{(j)}_1(x)) \in D$
  (whence $P^{(j)}_1 \in \Pi_\Sigma$), and
  \[
    f(x, P^{(j)}_{1}(x)) = -\pi_j(P^{(j)}_1(x)) = \min
    \{-\pi_j(\sigma) \mid \sigma \in \mathcal P_\Sigma(x)\} =
    -\max \pi_j(\mathcal P_\Sigma(x)),
  \]
  as desired. An analogous argument yields the desired $P^{(j)}_{0}$
  by taking $f(x,\sigma) = \pi_j(\sigma)$ rather than $f(x,\sigma) =
  -\pi_j(\sigma)$.

  Now, by \cref{pmeanlemma}, for any $\pi_\Sigma \in \Pi_\Sigma$ (and
  in particular, for $P^{(j)}_0, P^{(j)}_1$) the net barycentre field
  vanishes. So, using \eqref{eq:net-bary} we have
  \begin{align}
    0
    &= \int_{\mathbb{R}^d}|x - P^{(j)}_{0}(x)|^{p-2}(x -
      P^{(j)}_{0}(x)) d \mu(x) - \int_{\mathbb{R}^d}| x-
      P^{(j)}_{1}(x)|^{p-2}(x - P^{(j)}_{1}(x))d \mu(x) \nonumber \\
    &= \int_{\mathbb{R}^d} \dist(x, \Sigma)^{p-2}(P^{(j)}_{1}(x) -
      P^{(j)}_{0}(x)) d \mu(x), \label{eq:xy-diff-expr}
  \end{align}
  where the last equality comes from $|x- P^{(j)}_{1}(x)| =
  \dist(x,\Sigma) = |x - P^{(j)}_{0}(x)|$.

  Finally, suppose for the sake of contradiction that
  $\mu(\mathcal{A}_{\Sigma})> 0$. Observe that
  \[
    \mathcal{A}_{\Sigma} = \bigcup_{j=1}^d \{x \in \mathbb{R}^d \mid
    \pi_j \circ P_0^{(j)}(x) < \pi_j \circ P_1^{(j)}(x)\}.
  \]
  In particular, since $\mu(\mathcal{A}_{\Sigma}) > 0$, then for at
  least one $j \in \{1,\ldots,d\}$ the set
  \[
    S_j \coloneqq \{x \in \mathbb{R}^d \mid \pi_j \circ P_0^{(j)}(x) <
    \pi_j \circ P_1^{(j)}(x)\}
  \]
  has $\mu(S_j) > 0$. Fix such a $j$ and examine the
  $j$\textsuperscript{th} component of the integral in
  \eqref{eq:xy-diff-expr}. By construction we have $\pi_j \circ
  P_1^{(j)} \geq \pi_j \circ P_0^{(j)}$ everywhere, with the
  inequality strict for $x \in S_j$. Further, note that for $x \in
  S_j$ we have $\dist(x,\Sigma)^{p-2} > 0$, since otherwise $x \in
  \Sigma$ and so $P_0^{(j)}(x) = x = P_1^{(j)}(x)$. Thus the
  $j$\textsuperscript{th} component of \eqref{eq:xy-diff-expr} is
  strictly positive (and hence nonzero), a contradiction, and so
  $\mu(\mathcal A_\Sigma) = 0$ as desired.
\end{proof}

\begin{remark}\label{remark:afternegligibilityofridgeset}
  By \cref{ridgesetnegligable}, we see that for any minimizer $\Sigma
  \in \mathcal{S}_l$ of the \eqref{eq:adp}, the closest-point
  projection $\pi_{\Sigma}$ onto $\Sigma$ is unique $\mu$-a.e., and
  thus every $\pi_{\Sigma} \in \Pi_{\Sigma}$ gives the same barycentre
  field $\mathcal{B}_{\pi_{\Sigma}}$. In this case, we will abuse
  terminology and simply refer to $\mathcal{B}_{\pi_{\Sigma}}$ as
  \textit{the barycentre field of }$\Sigma$.
\end{remark}

\section{Topological properties of average distance
  minimizers}\label{sec:topologicalproperties}
In this section, we discuss the topological properties of minimizers
of the average distance problem \eqref{eq:adp}. Such properties have
been one of the main areas of study related to the \eqref{eq:adp}
since its introduction in \cite{Buttazzo02}. A complete topological
description of minimizers of \eqref{eq:adp} was given in
\cite{Buttazzo03} in two dimensions ($d=2$) with $p=1$, where it was
shown that optimal networks contain no loops (i.e.\ homeomorphic
images of $S^1$), have finitely many noncut points, and meet only in
triple junctions. This characterization relies on
\cite{Buttazzo03}*{Lemma 7.1}, which says that
\begin{equation}\label{eqn:atom-exists}
  \hbox{for  $d=2$ and any optimizer $\Sigma \in \mathcal{S}_l$,
    the measure $\nu = (\pi_{\Sigma})_{\#}\mu$ has an atom.}
\end{equation}
(Here, ``atom'' means a point $\sigma^* \in \Sigma$ with
$\nu(\{\sigma^*\}) >0$).

The fact that minimizers contain no loops was later shown to hold in
any dimension $d \geq 2$ \cite[Theorem 5.6]{Stepanov04}, and it was
shown in \cite{Stepanov06}*{Theorem 5.5} that
\begin{equation}\label{eqn:noncut-to-show-sec}
  \hbox{\small for $d\geq 2$ and $p \geq 1$, if
    \eqref{eqn:atom-exists} holds, then every noncut point of $\Sigma$
    is an atom of $\nu$}.
\end{equation}
Property \eqref{eqn:noncut-to-show-sec} plays a crucial role in
providing the desired topological description of the minimizers.
However, establishing \eqref{eqn:atom-exists} (and so
\eqref{eqn:noncut-to-show-sec}) for $d > 2$ proved to be difficult,
being described over a decade ago as an open problem ``of great
interest'' \cite{Lemenant12}.

It turns out that the barycentre field is a very useful tool for
studying when \eqref{eqn:noncut-to-show-sec} holds for $d \geq 2$; as
we show in \cref{subsec:barycentrenontrivialityandthekeylemma}, the
property \eqref{eqn:noncut-to-show-sec} holds if the barycentre field
of the optimizer is nontrivial. This nontriviality will be shown in
\cref{subsec:minimizershavenontrivialbarycentrefields} for all $d \geq
2$ when $p = 2$ or $p > \frac{1}{2}(3 + \sqrt{5})$, thus implying
\eqref{eqn:noncut-to-show-sec} and consequently topological
description of optimal $\Sigma$ for these cases; see
\cref{sec:topology}.

\subsection{Barycentre nontriviality and atomic noncut
  points}\label{subsec:barycentrenontrivialityandthekeylemma}

We show in this section that nontrivial barycentre field implies
\eqref{eqn:noncut-to-show-sec}: namely, that all noncut points are
atoms.

As a motivation for our argument, we first discuss an idea from the
proofs of \cite{Buttazzo03}*{Proposition 7.1} and
\cite{Stepanov06}*{Theorem 5.5}, which show that property
\eqref{eqn:noncut-to-show-sec} (namely, that all noncut points are
atoms of $\nu$) follows from property \eqref{eqn:atom-exists} (the
existence of an atom for $\nu$). Their idea is as follows. Let $\Sigma
\in \mathcal{S}_l$ be an optimal solution of the \eqref{eq:adp}, and
suppose we are given an atom $\sigma^* \in \Sigma$ of $\nu$, so
$\nu(\{\sigma^*\}) > 0$. Let $\sigma \in \Sigma$ be a noncut point
with $\sigma^* \ne \sigma$. By constructing a competitor $\Sigma' \in
\mathcal{S}_l$, we will use the optimality of $\Sigma$ to bound
$\nu(\{\sigma\})$ below by $\nu(\{\sigma^*\})$. First, we remove a
neighbourhood of radius $\epsilon$ centred at $\sigma$ from $\Sigma$
to produce a set $\Sigma_{\epsilon}$. This set $\Sigma_{\epsilon}$
recovers $\epsilon$ budget, while only increasing the objective value
$\mathscr{J}_p(\Sigma_{\epsilon})$ by something proportional to
$\epsilon \nu(\{\sigma\})$. Then, using the fact that $\sigma^*$ is an
atom, by adding a line segment to $\Sigma_{\epsilon}$ at $\sigma^*$ we
may construct a competitor $\Sigma'$ to $\Sigma$ which is better than
$\Sigma_{\epsilon}$ by something proportional to $\epsilon
\nu(\{\sigma^*\})$. So, in order to avoid contradicting the optimality
of $\Sigma$, we must have
\[
  \nu(\{\sigma\}) \geq C\nu(\{\sigma^*\})
\]
for some constant $C$ depending only on $\sigma^*$.

Our idea is to use the barycentre field instead of the atom $\sigma^*$
to construct a competitor using \cref{ifbarycentrenontrivial}. By the
same principle, this will allow us to bound $\nu(\{\sigma\})$ in terms
of the barycentre field, thus proving that all noncut points are atoms
if the barycentre is nontrivial. We begin by recalling the following
technical lemma from \cite{Buttazzo03}*{Lemma 6.1}.

\begin{lemma}[Noncut-neighbourhood lemma]\label{noncutneighbourhood}
  Let $\Sigma$ be a locally connected metric continuum (i.e.\ a
  compact, connected, and locally-connected metric space) containing
  more than one point, and let $\sigma \in \Sigma$ be a noncut point
  of $\Sigma$. Then, there exists a sequence $\{B_n\}_{n \in
    \mathbb{N}}$ of open subsets of $\Sigma$ satisfying the following
  conditions:
  \begin{itemize}
    \item[(i)] $\sigma \in B_n$ for all sufficiently large $n$,
    \item[(ii)] $\Sigma \setminus B_n$ is connected for each $n \in
      \mathbb{N}$,
    \item[(iii)] $\diam(B_n) \to 0$ as $n \to \infty$, and
    \item[(iv)] $B_n$ is connected for every $n$.
  \end{itemize}
\end{lemma}

Now, we provide a result that formalizes our discussion about
replacing the atom in the arguments of \cite{Buttazzo03}*{Proposition
  7.2} and \cite[Theorem 5.5]{Stepanov06} with the barycentre field.
\begin{theorem}[Bounding the mass of noncut
  points]\label{prop:boundingmassofnoncutpoints} Suppose $l > 0$.
  Let $p \geq 2$ and suppose that $\Sigma \in \mathcal{S}_l$ is an
  optimizer and that $\Sigma$ contains at least two points. Let
  $\pi_\Sigma \in \Pi_\Sigma$ and $\nu = (\pi_{\Sigma})_{\#}\mu$.
  Then, there exists some constant $\lambda > 0$ depending only on
  $\mathcal B_{\pi_\Sigma}$ such that for all noncut points $\sigma^*
  \in \Sigma$ we have
  \[
    \nu\{\sigma^*\}|\mathcal{B}_{\pi_{\Sigma}}(\sigma^*)| \geq
    \frac{\lambda }{4l} \int_{\Sigma}
    |\mathcal{B}_{\pi_{\Sigma}}(\sigma)|^2 d\nu(\sigma).
  \]
\end{theorem}

\begin{proof}
  Observe that if $\mathcal{B}_{\pi_\Sigma}(\sigma)$ is trivial
  (\cref{def:trivialbarycentrefield}) then the claim is trivial as
  well; hence suppose $\mathcal{B}_{\pi_\Sigma}(\sigma)$ is
  nontrivial.

  Let $\{B_n\}_{n \in \mathbb{N}}$ be as in
  \cref{noncutneighbourhood}. For each $n \in \mathbb{N}$, let
  $\epsilon_n = \frac{1}{2}\diam(B_{n})$. Define $\Sigma_n = \Sigma
  \setminus B_n$ and let $P_n(x)$ denote the closest-point projection
  onto $\Sigma \cap \partial B_n$. We now bound $\mathscr
  J_p(\Sigma_n)$ by decomposing the domain of integration into two
  pieces and bounding them separately. The first will be
  $\pi_\Sigma^{-1}(\Sigma \setminus \overline{B_n})$, on which by
  construction for all $x \in \pi_\Sigma^{-1}(\Sigma \setminus
  \overline{B_n})$ we have $d(x, \Sigma_n) = d(x, \Sigma)$. The second
  piece will be $\pi_\Sigma^{-1}(\overline{B_n})$, on which by
  construction we have $d(x, \Sigma_n) = d(x, P_n(x))$. In general $|x
  - P_n(x)| \leq |x - P_n(\pi_\Sigma(x))|$, so
  \begin{align*}
    \mathscr{J}_p(\Sigma_n)
    & = \int_{\pi_{\Sigma}^{-1}(\Sigma \setminus
      \overline{B_n})}\dist(x, \Sigma)^p d\mu(x) +
      \int_{\pi_{\Sigma}^{-1}(\overline{B_n})}|x- P_n(x)|^p d\mu(x)\\
    &\leq \int_{\pi_{\Sigma}^{-1}(\Sigma \setminus
      \overline{B_n})}\dist(x, \Sigma)^p d\mu(x) \\
    & \qquad + \int_{\pi_{\Sigma}^{-1}(\overline{B_n})}|x-
      \pi_{\Sigma}(x) - (P_n(\pi_{\Sigma}(x)) - \pi_{\Sigma}(x))|^p
      d\mu(x).
  \end{align*}
  For the last term, since $p \ge 2$, applying
  \cref{lemma:basicinequality} gives
  \begin{align*}
    & |x - \pi_{\Sigma}(x)
      -(P_{n}(\pi_{\Sigma}(x)) - \pi_{\Sigma}(x))|^p - |x -
      \pi_{\Sigma}(x)|^p \\
    &\leq \frac{p}{2}( |x - \pi_{\Sigma}(x) -(P_{n}(\pi_{\Sigma}(x))
      - \pi_{\Sigma}(x))|^2 - |x - \pi_{\Sigma}(x)|^2
      )|x-\pi_{\Sigma}(x)|^{p-2} \\
    &= \frac{p}{2}\Big(
      |P_{n}(\pi_{\Sigma}(x)) - \pi_{\Sigma}(x)|^2
      - 2(P_{n}(\pi_{\Sigma}(x)) - \pi_{\Sigma}(x))\cdot
      (x - \pi_{\Sigma}(x)) \Big)|x - \pi_{\Sigma}(x)|^{p-2},
  \end{align*}
  and since for all $x \in \pi_\Sigma^{-1}( \overline{B_n})$ we have
  $|P_{n}(\pi_{\Sigma}(x)) - \pi_{\Sigma}(x)| \leq \epsilon_n$,
  defining $M = \diam(\supp(\mu))$ and recalling the definition of the
  barycentre field $\mathcal{B}_{\pi_{\Sigma}}$ yields
  \begin{align}\label{eq:Sigma-n-Sigma} \nonumber
    \mathscr{J}_p(\Sigma_{n}) \leq
    & \mathscr{J}_p(\Sigma) - p\int_{\pi_{\Sigma}^{-1}(
      \overline{B_n})}(P_{n}(\pi_{\Sigma}(x)) - \pi_{\Sigma}(x))\cdot
      (x - \pi_{\Sigma}(x))|x - \pi_{\Sigma}(x)|^{p-2} d\mu(x) \\
    \nonumber
    &+ \frac{p}{2} \epsilon_n^2 M^{p-2} \nu(\overline{B_n}) \\
    \nonumber
    =& \mathscr{J}_p(\Sigma) - \int_{ \overline{B_n}}(P_{n}(\sigma) -
       \sigma)\cdot \mathcal{B}_{\pi_{\Sigma}}(\sigma) d\nu(\sigma) +
       \frac{p}{2} \epsilon^2 M^{p-2}\nu(\overline{B_n}) \\
    \leq
    & \mathscr{J}_p(\Sigma) + \epsilon_n
      \int_{\overline{B_n}}|\mathcal{B}_{\pi_{\Sigma}}(\sigma)|
      d\nu(\sigma) +\frac{p}{2} \epsilon_n^2
      M^{p-2}\nu(\overline{B_n}).
  \end{align}

  On the other hand, by \cref{approximatebysmooth} there exists a
  Lipschitz map $\xi: \Sigma \to \mathbb{R}^d$ such that
  \begin{equation}
    \int_{\Sigma}\xi(\sigma)\cdot
    \mathcal{B}_{\pi_{\Sigma}}(\sigma)d\nu(\sigma) >
    \frac{1}{2}\int_{\Sigma}|\mathcal{B}_{\pi_{\Sigma}}(\sigma)|^2
    d\nu(\sigma) > 0. \label{eq:applying-lem-2-11}
  \end{equation}
  Let $L >0$ be a Lipschitz constant for $\xi$, and let $\lambda =
  \max\{\frac{1}{L}, \frac{1}{\max |\xi|}\}$; note $\lambda$ does not
  depend on $\sigma^*$. Then $\lambda \xi$ is $1$-Lipschitz, and thus
  $\sigma \mapsto \sigma + \frac{\epsilon_n}{(l - \epsilon_n)}\lambda
  \xi(\sigma)$ is $1 + \frac{\epsilon_n}{l - \epsilon_n}$-Lipschitz.
  So,
  \[
    \Sigma_n' \coloneqq \{\sigma + \frac{\epsilon_n}{1 -
      \epsilon_n}\lambda \xi(\sigma) \mid \sigma \in \Sigma_{n}\} \in
    \mathcal{S}_l.
  \]

  Now, we want to estimate $\mathscr{J}_p(\Sigma_n') -
  \mathscr{J}_p(\Sigma_n)$ in terms of $\int_{\Sigma}\xi(\sigma)\cdot
  \mathcal{B}_{\pi_{\Sigma}}(\sigma)d\nu(\sigma)$. To do this, we
  first will estimate the difference
  \[
    \bigg|\int_{\Sigma_n} \xi(\sigma)\cdot
    \mathcal{B}_{\pi_{\Sigma_n}}(\sigma)d\nu_{\pi_{\Sigma_n}}(\sigma)
    - \int_{\Sigma \setminus \{\sigma^*\}}\xi(\sigma)\cdot
    \mathcal{B}_{\pi_{\Sigma}}d\nu(\sigma)\bigg|,
  \]
  where we recall that $\nu_{\pi_{\Sigma_n}} \coloneqq
  (\pi_{\Sigma_n})_{\#}\mu$.

  By \cref{ridgesetnegligable} we have $\mu(\mathcal{A}_{\Sigma}) =
  0$, and so for $\mu$-a.e.\ $x\in\mathbb{R}^d \setminus
  \pi_{\Sigma}^{-1}(\overline{B_n})$ we get $\pi_{\Sigma}(x) =
  \pi_{\Sigma_n}(x)$. For the $x \in \pi_\Sigma^{-1}(\overline{B_n})$,
  note $\Sigma_n \subseteq \Sigma$ gives the uniform bound
  \[
    |\pi_\Sigma(x) - \pi_{\Sigma_n}(x)| \leq \mathrm{diam}(\Sigma)
    \leq l.
  \]
  Next, note that regardless of the choices of $\Sigma$ and $x$ we get
  $|\mathcal B_{\pi_\Sigma}(x)| \leq p M^{p-1}$. So the Lipschitz
  condition on $\xi$ gives
  \begin{align*}
    & \bigg|\int_{\Sigma_n} \xi(\sigma) \cdot
      \mathcal{B}_{\pi_{\Sigma_n}}(\sigma)d\nu_{\pi_{\Sigma_n}}(\sigma)
      - \int_{\Sigma \setminus \{\sigma^*\}} \xi(\sigma)\cdot
      \mathcal{B}_{\pi_{\Sigma}} d\nu(\sigma)\bigg| \\
    &\leq pLM^{p-1}\int_{\mathbb{R}^d \setminus \{\sigma^*\}}
      |\pi_{\Sigma}(x) - \pi_{\Sigma_n}(x)| d\mu(x)  \\
    &\leq pLM^{p-1}l\nu(\overline{B_n}\setminus \{\sigma^*\}),
  \end{align*}
  which is $o(1)$. Thus, in particular,
  \begin{equation}
    \varepsilon_n \bigg|\int_{\Sigma_n} \xi(\sigma) \cdot
    \mathcal{B}_{\pi_{\Sigma_n}}(\sigma)d\nu_{\pi_{\Sigma_n}}(\sigma)
    - \int_{\Sigma \setminus \{\sigma^*\}} \xi(\sigma)\cdot
    \mathcal{B}_{\pi_{\Sigma}} d\nu(\sigma)\bigg|  =
    o(\epsilon_n). \label{eq:bary-diff-bound}
  \end{equation}
  Now, by \cref{strongbarycentreapproximation} we have
  \begin{align}
    \mathscr{J}_p(\Sigma_n') - \mathscr{J}_p(\Sigma_n)
    &\leq -\frac{\epsilon_n}{l-\epsilon_n} \int_{\Sigma_n}
      \lambda\xi(\sigma)\cdot
      \mathcal{B}_{\pi_{\Sigma_n}}(\sigma)d\nu_{\pi_{\Sigma_n}}(\sigma)
      + o(\epsilon_n),
      \shortintertext{whence \eqref{eq:bary-diff-bound} gives}
      \mathscr{J}_p(\Sigma_n') -
      \mathscr{J}_p(\Sigma_n)
    &\leq -\frac{\epsilon_n}{l-\epsilon_n} \int_{\Sigma \setminus
      \{\sigma^*\}} \lambda\xi(\sigma)\cdot \mathcal{B}_{\pi_{\Sigma}}
      (\sigma) d\nu(\sigma) + o(\epsilon_n) \nonumber \\
    & = -\frac{\epsilon_n}{l-\epsilon_n} \int_{\Sigma}
      \lambda\xi(\sigma)\cdot \mathcal{B}_{\pi_{\Sigma}} (\sigma)
      d\nu(\sigma) \nonumber \\
    &\qquad + \epsilon_n \nu(\{\sigma^*\}) \lambda \xi(\sigma^*) \cdot
      \mathcal B_{\pi_\Sigma}(\sigma^*) + o(\epsilon_n) \nonumber \\
    &\leq -\frac{\epsilon_n}{l-\epsilon_n}\int_\Sigma \lambda
      \xi(\sigma) \cdot \mathcal B_{\pi_\Sigma}(\sigma) d\nu(\sigma)
      \nonumber \\
    &\qquad + \epsilon_n\nu(\{\sigma^*\}) |\mathcal
      B_{\pi_\Sigma}(\sigma^*)| +
      o(\epsilon_n). \label{eq:Sigma-prime-Sigma-n}
  \end{align}
  Note that the error term in \eqref{eq:Sigma-n-Sigma} is
  $o(\epsilon_n^2)$. So, adding \eqref{eq:Sigma-n-Sigma},
  \eqref{eq:Sigma-prime-Sigma-n} and then applying
  \eqref{eq:applying-lem-2-11} we have
  \begin{align*}
    \mathscr{J}_p(\Sigma_n') - \mathscr{J}_p(\Sigma)
    &< \epsilon_n \bigg(\int_{\overline{B_n}}
      |\mathcal{B}_{\pi_{\Sigma}}(\sigma)|d\nu(\sigma) -
      \frac{1}{2(l-\epsilon_n)}\lambda \int_{\Sigma \setminus
      \{\sigma^*\}} |\mathcal{B}_{\pi_{\Sigma}}(\sigma)|^2
      d\nu(\sigma) \\
    &\quad \qquad + \nu(\{\sigma^*\}) |\mathcal
      B_{\pi_\Sigma}(\sigma^*)| \bigg) + o(\epsilon_n).
  \end{align*}
  Since $\Sigma$ was assumed to be optimal, $\mathscr{J}_p(\Sigma_n') -
  \mathscr{J}_p(\Sigma) \geq 0$, and so for all sufficiently large $n$
  we have
  \[
    \int_{\overline{B_n}}|\mathcal{B}_{\pi_{\Sigma}}(\sigma)|d\nu(\sigma)
    - \frac{\lambda}{2(l-\epsilon_n)} \int_{\Sigma }
    |\mathcal{B}_{\pi_{\Sigma}}(\sigma)|^2 d\nu(\sigma) +
    \nu(\{\sigma^*\}) |\mathcal B_{\pi_\Sigma}(\sigma^*)| \geq 0.
  \]
  Decomposing the left integral via $\overline{B_n} = \{\sigma^*\}
  \cup (\overline{B_n} \setminus \{\sigma^*\})$ and taking $n \to
  \infty$ thus yields that
  \[
    \nu\{\sigma^*\}|\mathcal{B}_{\pi_{\Sigma}}(\sigma^*)| \geq
    \frac{\lambda }{4l}\int_{\Sigma }
    |\mathcal{B}_{\pi_{\Sigma}}(\sigma)|^2 d\nu(\sigma),
  \]
  as desired.
\end{proof}

Finally, we are ready to prove that $\Sigma$ having nontrivial
barycentre field implies noncut points are atoms of $\nu$.
\begin{corollary}
  \label{cor:nontrivialequivatom}
  Suppose $p \geq 2$. Let $\Sigma \in \mathcal{S}_l$ be a minimizer of
  the \eqref{eq:adp}, and let $\pi_\Sigma\in \Pi_\Sigma$ and $\nu =
  (\pi_{\Sigma})_{\#}\mu$. If the barycentre field
  $\mathcal{B}_{\pi_\Sigma}$ is nontrivial, then every noncut point
  $\sigma^* \in \Sigma$ is an atom; that is, $\nu (\{\sigma^*\})>0.$
\end{corollary}

\begin{proof}
  Since $\mathcal{B}_{\pi_{\Sigma}}$ is nontrivial
  (\cref{def:trivialbarycentrefield}) we have $\int_\Sigma |\mathcal
  B_{\pi_\Sigma}(\sigma)|^2 \, d\nu(\sigma) > 0$, whence
  \cref{prop:boundingmassofnoncutpoints} immediately yields the
  result.
\end{proof}
\begin{remark}
  In fact, the converse of \cref{cor:nontrivialequivatom} holds:
  namely, under \eqref{assum:zero-mu}, if there is an atom $\sigma^*$
  for $\nu$ then the barycentre field is nontrivial. This will be
  shown in \cref{subsec:minimizershavenontrivialbarycentrefields}; see
  \cref{rmk:atom-nontrivialbary}. The intuition for this is as
  follows. Recall from \eqref{eq:fbary-grad-J-min} that the barycentre
  field encodes when we are able to get an improvement of order
  $\epsilon$ to the objective value $J(l)$ \eqref{eq:jdefinition}
  given $\epsilon$ additional budget. However, if for some minimizer
  $\Sigma$ the measure $\nu = (\pi_{\Sigma})_{\#}\mu$ has an atom, we
  can improve the objective value by $O(\epsilon)$ by adding a
  suitable segment of length $\epsilon$ at the atom. So we expect the
  barycentre field of any optimal $\Sigma$ to be nontrivial if there
  is an atom.
\end{remark}

\subsection{Minimizers have nontrivial barycentre
  fields}\label{subsec:minimizershavenontrivialbarycentrefields}

From \cref{cor:nontrivialequivatom}, we see that if the optimizer has
a nontrivial barycentre field then \eqref{eqn:noncut-to-show-sec}
holds, namely, all its noncut points are atoms. We prove in
\cref{nontrivialbarycentre} below that for certain values of $p$, the
barycentre field of any optimizer is nontrivial, thus establishing
\eqref{eqn:noncut-to-show-sec} in these cases. A similar result was
previously obtained for the special $p=2$ case of the
\emph{length-constrained principal curves} problem in
\cite{Delattre17}*{Lemma 3.2}, and our proof takes inspiration from
their approach (see the overview from
\cref{sec:bary-nontrivial-summary}).

\begin{theorem}[Minimizers have nontrivial barycentre
  fields]\label{nontrivialbarycentre} Assume \eqref{assum:zero-mu}.
  Suppose $l > 0$. Let $\Sigma \in \mathcal{S}_l$ be a solution of the
  \eqref{eq:adp}. Assume $p = 2$ or $p > \lbonpval$. Then,
  $\pi_{\Sigma}$ has nontrivial barycentre field
  $\mathcal{B}_{\pi_\Sigma}$.
\end{theorem}
The proof of this theorem is subtle and long and is given in
\cref{sec:proof-main}.

\begin{remark}
  When $l = 0$, the barycentre field of an optimizer is necessarily
  trivial. To see this, note that any $\Sigma \in \mathcal{S}_0$ is
  necessarily a singleton, whence $\mathcal B_{\pi_\Sigma} \equiv
  \mathcal B_{\pi_\Sigma}^{\rm net}$. But by \cref{pmeanlemma}, if
  $\Sigma$ is optimal, then $\mathcal B_{\pi_\Sigma}^{\rm net} = 0$,
  whence $\mathcal B_{\pi_\Sigma}$ as well.
\end{remark}

\begin{corollary}[Right-derivative bound for
  $J$]\label{epsilonimprovement} Assume \eqref{assum:zero-mu} and that
  $p = 2$ or $p > \frac{1}{2}(3 + \sqrt{5})$. Then, for each $l > 0$,
  there exists some $C > 0$ such that
  \[
    \lim_{\epsilon \to 0^+}\frac{J(l + \epsilon) - J(l)}{\epsilon}
    \leq -C.
  \]
\end{corollary}

\begin{proof}
  This follows from \cref{nontrivialbarycentre} and
  \cref{ifbarycentrenontrivial}.
\end{proof}

\begin{remark} Intuitively, we expect that the barycentre field should
  be nontrivial for any $d > 2$ and $p \geq 1$, at least under the
  condition that $\mu \ll \mathrm{Leb}$. However, for smaller values
  of $p$, proving that the barycentre field is nontrivial seems to be
  significantly more difficult. This is because our method for the
  proof of \cref{nontrivialbarycentre} requires estimating the
  higher-order corrections to the quantity $J(l + \epsilon) - J(l)$, a
  process which depends significantly on the regularity of the
  derivatives of the integrand $|\cdot|^p$ near the origin.
\end{remark}

\subsection{Topological description of average distance
  minimizers}\label{sec:topology}

We conclude by combining our results with Stepanov's partial result
\cite{Stepanov06}*{Theorem 5.5} to provide a complete topological
description of minimizers of \eqref{eq:adp} for $p = 2$ or $p >
\frac{1}{2}(3 + \sqrt{5})$. Recall that
\cref{topologicalcharacterization} (i) was already shown to hold
unconditionally in \cite{Stepanov04}.
\begin{theorem}\label{topologicalcharacterization} Denote
  $\mathrm{ord}_{\sigma}\Sigma = \inf_{\epsilon > 0}\card(\partial
  B_{\epsilon}(\sigma) \cap \Sigma )$, where $\mathrm{card}( \cdot )$
  denotes set cardinality. Assume \eqref{assum:zero-mu}. Let $p = 2$
  or $p > \lbonpval$, and let $\Sigma \in \mathcal{S}_l$ be optimal.
  Then:
  \begin{enumerate}[label=(\roman*)]
    \item (See \cite[Theorem 5.6]{Stepanov04}). $\Sigma$ does not
      contain any simple closed curves (homeomorphic images of $S^1$).
      In particular, every noncut point $\sigma \in \Sigma$ is an
      ``endpoint,'' i.e.\ $\mathrm{ord}_\sigma \Sigma = 1$.
    \item $\Sigma$ has finitely-many noncut points.
    \item $\Sigma$ has finitely-many ``branching points,'' i.e.\
      points $\sigma$ such that $\mathrm{ord}_{\sigma}\Sigma > 2$.
    \item Every branching point $\sigma \in \Sigma$ is a ``triple
      point,'' i.e.\ $\mathrm{ord}_{\sigma}\Sigma = 3$.
  \end{enumerate}
\end{theorem}
\begin{proof}
  By \cite{Stepanov06}*{Theorem 5.5}, (i) to (iv) holds for a
  minimizer $\Sigma \in \mathcal{S}_l$ of the \eqref{eq:adp}, under
  the assumption that $\nu = (\pi_{\Sigma})_{\#}\mu$ has an atom. If
  $l = 0$, this is trivially the case since $\nu$ has total mass 1, so
  suppose $l > 0$. By \cref{cor:nontrivialequivatom}, $\nu$ has an
  atom whenever the barycentre field $\mathcal{B}_{\pi_{\Sigma}}$ is
  nontrivial, since by \cite{kuratowski}*{§47 Theorem IV.5} any
  $\Sigma \in \mathcal{S}_l$ has at least two noncut points. But
  \cref{nontrivialbarycentre} shows that the barycentre field is
  nontrivial for $p = 2$ or $p > \lbonpval$, thus (i) to (iv) hold in
  this case as well.
\end{proof}

\section{Proof of \cref{nontrivialbarycentre}}
\label{sec:proof-main}

Finally, we give a proof of \cref{nontrivialbarycentre}.

\subsection{Proof roadmap.} \label{sec:proof-roadmap}

We first outline the intuition for the argument; as mentioned
previously, our proof is inspired by the approach of
\cite{Delattre17}*{Lemma 3.2}.

\subsubsection{Intuition.} \label{sec:proof-intuition}
From a high-level perspective, the idea is the following. Let
\begin{equation}
  \label{eqn:sigma-optimal}
  \text{$\Sigma \in \mathcal{S}_l$ be an optimizer of
    \eqref{eq:objective-minimization},}
\end{equation}
and for the sake of contradiction suppose that
\begin{equation}
  \label{eqn:barycentrezero}
  \text{$\pi_{\Sigma}$ has trivial barycentre field
    $\mathcal{B}_{\pi_\Sigma} \equiv 0$.}
\end{equation}
First, we shrink $\Sigma$ to obtain a $\Sigma_\epsilon$ that recovers
$O(\epsilon)$ budget; by triviality of the barycentre field and the
fact that $p \geq 2$, we will get (\cref{rem:error-estimates})
\[
  \mathscr{J}_p(\Sigma_\epsilon) -
  \mathscr J_p(\Sigma) = O(\epsilon^2).
\]
Second, we use the $O(\epsilon)$ recovered budget to modify
$\Sigma_\epsilon$ to obtain a $\Sigma^* \in \mathcal S_l$ with
\[
  \mathscr{J}_p(\Sigma_\epsilon) - \mathscr{J}_p(\Sigma^*) =
  O(\epsilon^\alpha),
\]
where $\alpha < 2$. For a particular choice of $\epsilon$ sufficiently
small, this yields $\mathscr J_p(\Sigma^*) < \mathscr J_p(\Sigma)$, a
contradiction to optimality of $\Sigma$.

\subsubsection{Organization of the proof} Our actual proof differs
slightly from the intuition above, in that we estimate $\mathscr
J_p(\Sigma) - \mathscr J_p(\Sigma^*)$ directly rather than separately
estimating $\mathscr J_p(\Sigma_\epsilon) - \mathscr J_p(\Sigma)$ and
$\mathscr J_p(\Sigma_\epsilon) - \mathscr J_p(\Sigma^*)$. We found
that this approach yielded a proof which encodes the same ideas, but
cuts down on the number of technical estimates required. However, even
with this simplification, the proof is nontrivial, and requires
delicate analysis. There are three main steps.

First (\cref{sec:first-main-step}), via
\crefrange{p=2lowerbound}{lem:psi-lowerbound} we obtain a technical
lower bound (\cref{cor:step-one-summary}) for $\mathscr J_p(\Sigma) -
\mathscr J_p(\Sigma^*)$ that holds independently of whether certain
parameters were chosen favourably. The main parameters in question are
a point $\sigma^* \in \Sigma$ (used in the construction of $\Sigma^*$
in \cref{sec:competitor-construction}) and a set $A \subseteq \Sigma$.

Second (\cref{sec:second-main-step}), in \cref{badsetbig,badsetnull},
we show that there exist favourable ways to choose $\sigma^*$ and $A$
that allow us to make quantitative refinements to the bounds from the
first step. In particular, given scalars $0 \leq s < 1$ and $K \geq
0$, we define a set $B_{K}^s$ that loosely encodes points of $\Sigma$
where $\nu$ has ``local dimension'' less than or equal to $s$.
Roughly, from the disintegration theorem we may expect that points
$\sigma \in B_{K}^s$ have fibres under the closest-point projection
containing greater-than-average mass from $\mu$. Then, in
\cref{badsetbig} we show that when $\nu(B_K^s) > 0$ one may pick
$\sigma^* \in B_K^s$ so that the fibres of points in a neighbourhood
of $\sigma^*$ give an average-or-better contribution to $\mathscr J_p$
than what is typical on $B_K^s$. The case $\nu(B_k^s) = 0$
(\cref{badsetnull}) is more subtle, but the general idea is still to
try to find a $\sigma$ where local contributions to $\mathscr J_p$
from fibres are average-or-better. In either case, we obtain
quantitative bounds with decay rates depending on $s$, plus an
additional parameter $q$ in \cref{badsetnull}.

Third (\cref{sec:finalsteps}), we show that provided $p=2$ or $p >
\frac{3 + \sqrt 5}{2}$, there exist choices of $s$, $K$, and (when
applicable) $q$ such for all $\epsilon$ sufficiently small, the decay
rates of the previous step yield $\mathscr J_p(\Sigma) - \mathscr
J_p(\Sigma^*) > 0$, contradicting optimality of $\Sigma$
\eqref{eqn:sigma-optimal}, thus proving our claim.

In any case, the key arguments in the first two steps rely on the
specific construction of the competitor $\Sigma^*$, hence we detail it
now.

\subsubsection{Construction of the competitor
  $\Sigma^*$} \label{sec:competitor-construction} The general
construction for $\Sigma^*$ is as follows. First, fix some $\sigma^*
\in \Sigma$. Since $\mathscr J_p$ and $\mathscr C$ depend only on the
metric structure of $\supp(\mu) \cup \Sigma$, observe that
\begin{equation}
  \text{without loss of generality, we may use coordinates
    having $\sigma^*$ at the origin.}
  \label{coordinate-translation}
\end{equation}
Next, fix an arbitrary $\epsilon \in (0,1)$, and consider
$\Sigma_{\epsilon} = (1 - \epsilon)\Sigma$; observe that
$\Sigma_\epsilon$ recovers $l\epsilon$ budget (in particular, we have
$\Sigma_\epsilon \in \mathcal{S}_{(1-\epsilon) l}$) while keeping
$\mathbf{0} = \sigma^* \in \Sigma_\epsilon$. For $\tau > 0$, define
the $d$-dimensional ``cross shape'' $S_\tau$ by
\[
  S_{\tau} = \bigcup_{i=1}^d \{e_i t \mid t \in [-\tau, \tau]\}.
\]
This perturbation was previously considered in the context of the
\eqref{eq:adp} by Paolini and Stepanov in \cite{Stepanov04}*{Lemma
  3.3}, and in the context of length-constrained principal curves by
Delattre and Fischer in \cite[Lemma 3.1]{Delattre17}. Anyways, observe
that taking
\begin{align}\label{eqn:tau-alpha}
  \tau = \alpha\epsilon \qquad \text{ where} \qquad \alpha =
  \frac{l}{2d}
\end{align}
guarantees $\mathscr C(S_\tau) = \epsilon$.

With this, we define the competitor
\begin{align}\label{eqn:competitor}
  \Sigma^* \coloneqq
  (\Sigma_{\epsilon}\cup S_{\tau}) \in \mathcal S_l.
\end{align}
Observe that we have suppressed the dependence of $\Sigma^*$ on
$\sigma^*$ and $\epsilon$ in our notation.

\begin{remark}
  With the simple tweak $\alpha = \frac{l}{4d}$ in the construction of
  the competitor, the same proof that we present below for
  \cref{nontrivialbarycentre} works in the case of
  \emph{length-constrained principal curves} (see e.g.\
  \cite{Delattre17}). The extra factor of $2$ comes from the fact that
  we need to parameterize the $d$-dimensional cross, and thus each arm
  of the cross adds twice as much length as in the case of the
  $\mathcal{H}^1$ constraint.
\end{remark}

\subsection{First step: lower bounding $\mathscr
  J_p(\Sigma) - \mathscr J_p(\Sigma^*)$} \label{sec:first-main-step}

Now, we proceed with the proof. We use the notation $|x|_{\infty} =
\max_{i \in \{1, \dots, d\}}|x_i|$ for $x \in \mathbb{R}^d$. As in the
roadmap section above, fix $\sigma^* \in \Sigma$, $\epsilon \in
(0,1)$, and recall $\Sigma^*$ from \eqref{eqn:competitor}.
Lastly, let
$A \subseteq \Sigma$ be an arbitrary neighbourhood of $\sigma^*$; we
will choose $A$ in a favourable way later.

\subsubsection{Two general lemmata}
We begin with \cref{p=2lowerbound,lem:lower-bound-general}, which hold
in general, regardless of the optimality of $\Sigma$ (see
\eqref{eqn:sigma-optimal}) or whether the barycentre field is trivial
or not (see \eqref{eqn:barycentrezero}).
\begin{lemma}\label{p=2lowerbound}
  Using the coordinate system with $\sigma^* = 0$
  (see \eqref{coordinate-translation}), construct $\Sigma^*$ as in
  \eqref{eqn:competitor}. Then for all $x \in \mathbb{R}^d$, we have
  \[
    \dist(x, \Sigma)^2 - \dist(x, \Sigma^*)^2 \geq \psi(x),
  \]
  where $\psi (x) =\min [ \psi_1(x), 0]$ with
  \begin{align}\label{eqn:psi-def}
    \psi_1(x) =
    \max [
    & -\!2\epsilon \pi_{\Sigma}(x)\cdot (x- \pi_{\Sigma}(x)) -
      \epsilon^2 |\pi_{\Sigma}(x)|^2, \\\nonumber
    &
      -\!2\phantom{\epsilon}\pi_{\Sigma}(x)\cdot (x - \pi_{\Sigma}(x))
      - \phantom{\epsilon^2} |
      \pi_{\Sigma}(x)|^2 + 2 \tau|x |_{\infty} - \tau^2].
  \end{align}
\end{lemma}
Before presenting the proof, we note here that the terms $(x -
\pi_{\Sigma}(x))$ in \eqref{eqn:psi-def} will be important in
\cref{lem:psi-integral} to obtain a relation to the barycentre field
$\mathcal{B}_{\pi_\Sigma}$.
\begin{proof}
  We derive the lower bound involving $\psi$. First, since $\dist(x,
  \Sigma^*)^2\le |x -\pi_{\Sigma_\epsilon}(x)|^2 $,
  \begin{align}\label{eqn:psi-1} \nonumber
    \dist(x, \Sigma)^2 - \dist(x, \Sigma^*)^2
    &\geq |x - \pi_{\Sigma}(x)|^2 - |x -
      (1-\epsilon)\pi_{\Sigma}(x)|^2 \\\nonumber
    &= - 2\epsilon x \cdot \pi_\Sigma(x) + (2\epsilon - \epsilon^2)
      |\pi_\Sigma(x)|^2 \\
    &= - 2 \epsilon \pi_{\Sigma}(x)\cdot (x - \pi_{\Sigma}(x)) -
      \epsilon^2 |\pi_{\Sigma}(x)|^2.
  \end{align}
  Next, we want to estimate $|x - \pi_{S_\tau}(x)|^2$. We have two
  subcases.

  First suppose $|x|_\infty \leq \tau$. Since each $v \in S_\tau$ has
  at most one nonzero component, we see $\min_{v \in S_\tau} |x - v|$
  is achieved by taking $v$ to be the largest component of $x$:
  \[
    |x - \pi_{S_\tau}(x)|^2 = \min_{v \in S_\tau} \sum_i |x_i - v_i
    |^2 = \Big(\sum_i |x_i|^2\Big) - |x|_\infty^2 = |x|^2 -
    |x|_\infty^2.
  \]
  Second, suppose $|x|_\infty > \tau$. The same reasoning shows the
  optimal $v$ is the endpoint of the cross arm that points in the
  direction of the largest component of $x$, whence
  \[
    |x - \pi_{S_\tau}(x)|^2 = |x|^2 - |x|_\infty^2 + (|x|_\infty -
    \tau)^2.
  \]

  We may combine the two subcases by writing
  \[
    |x - \pi_{S_{\tau}}(x)|^2 = |x|^2 + ((|x|_{\infty} - \tau)_+)^2 -
    |x|_{\infty}^2.
  \]
  Since $((|x|_\infty - \tau)_+)^2 \leq (|x|_\infty - \tau)^2$
  we obtain
  \[
    |x - \pi_{S_\tau}(x)|^2 \leq |x|^2 - 2 \tau |x|_\infty + \tau^2.
  \]
  So using $\dist(x, \Sigma^*)^2\le |x - \pi_{S_\tau}(x)|^2 $,
  \begin{align*}
    \dist(x, \Sigma)^2 - \dist(x, \Sigma^*)^2
    &\geq |x - \pi_{\Sigma}(x)|^2 - |x - \pi_{S_{\tau}}(x)|^2 \\
    &\geq |x - \pi_{\Sigma}(x)|^2 - |x|^2 + 2 \tau |x|_{\infty} -
      \tau^2 \shortintertext{whence expanding $|x -
      \pi_\Sigma(x)|^2$ yields}
    &= - 2 x\cdot \pi_{\Sigma}(x) + |\pi_{\Sigma}(x)|^2 + 2 \tau
      |x|_{\infty} - \tau^2 \\
    &= -2(x-\pi_{\Sigma}(x))\cdot \pi_{\Sigma}(x) -
      |\pi_{\Sigma}(x)|^2 +  2 \tau |x|_{\infty} - \tau^2.
  \end{align*}
  Combining this with \eqref{eqn:psi-1} we get the desired lower
  bound.
\end{proof}
Now, we use \cref{p=2lowerbound} to find a lower bound on
$\mathscr{J}_p(\Sigma) - \mathscr{J}_p(\Sigma^*)$.
\begin{lemma}\label{lem:lower-bound-general}
  Assume \eqref{assum:zero-mu}. Then for $p \ge 2$ and $\psi$ defined
  as in \cref{p=2lowerbound}, we have
  \begin{equation}\label{equationlb}
    \mathscr{J}_p(\Sigma) - \mathscr{J}_p(\Sigma^*) \geq
    \frac{p}{2}\int_{\mathbb{R}^d}\psi(x)\dist(x, \Sigma)^{p-2}
    d\mu(x) + \frac{p}{2}\int_{\mathbb{R}^d}\psi(x)\zeta(x)d\mu(x),
  \end{equation}
  where
  \begin{align*}
    \zeta(x) \coloneqq
    \begin{cases}
      0,
      & p =2, \\
      -|\dist(x, \Sigma) - \dist(x, \Sigma^*)|^{p-2},
      & 2 < p < 3, \\
      (p-2)(\dist(x, \Sigma^*) - \dist(x, \Sigma)) \dist(x,
      \Sigma)^{p-3},
      & p \geq 3.
    \end{cases}
  \end{align*}
\end{lemma}
\begin{proof}
  Below we take the convention that $\dist(x,\Sigma)^0 \equiv 1$, even
  when $\dist(x,\Sigma)=0$; by \eqref{assum:zero-mu}, this will cause
  no problems. For $p \ge 2$, using the inequality
  \eqref{lemma:basicinequality}, namely, $a^p - b^p \geq
  \frac{p}{2}(a^2 - b^2)b^{p-2}$ (recall the convention $0^0=1$
  there), we have
  \begin{align*}
    \dist(x, \Sigma)^p - \dist(x, \Sigma^*)^p
    & \geq \frac{p}{2}(\dist(x, \Sigma)^2 - \dist(x,
      \Sigma^*)^2)\dist(x,\Sigma^*)^{p-2}\\
    & \ge \frac{p}{2}  \psi (x)  \dist(x,\Sigma^*)^{p-2}
  \end{align*}
  for $\psi$ from \cref{p=2lowerbound}.
  From this the $p=2$ case of \eqref{equationlb} follows.
  It remains to show that for the remaining case of $p$,
  \begin{equation}
    \dist(x, \Sigma^*)^{p-2} \geq \dist(x, \Sigma)^{p-2} +
    \zeta(x)
    \label{eq:bound-where-zeta-appears}
  \end{equation}

  First, suppose $2 < p < 3$, and define $f(t) = |t|^{p-2}$.
  Note that $f$ satisfies $f(a+b) \le f(a)+f(b)$. Since $f$ is also
  increasing on $[0,\infty)$, for all $c \in [0,a+b]$ we further
  obtain $f(a) + f(b) \geq f(c)$. We apply this fact with $a =
  \dist(x, \Sigma^*)$, $b = |\dist(x, \Sigma)- \dist(x, \Sigma^*)|$,
  and $c = \dist(x,\Sigma)$. To verify, we have
  \[
    a + b = \dist(x, \Sigma^*) + |\dist(x, \Sigma)- \dist(x,
    \Sigma^*)| \geq \dist(x, \Sigma) = c,
  \]
  whence from $f(a) \geq f(c) - f(b)$ we get
  \[
    \dist(x, \Sigma^*)^{p-2} \geq \dist(x, \Sigma)^{p-2} - |\dist(x,
    \Sigma^*) - \dist(x, \Sigma)|^{p-2}.
  \]
  This gives \eqref{eq:bound-where-zeta-appears} with the desired form
  of $\zeta(x)$ for the $2 < p < 3$ case.

  Now, suppose $p \geq 3$. To obtain
  \eqref{eq:bound-where-zeta-appears} we apply the lower bound of
  \eqref{lemma:basicinequality}, this time with ``$p$'' $= p-2$, and
  $q=1$ (note that $p \geq 3$ is necessary to satisfy inequality's
  hypothesis on the exponents). Explicitly, this gives $a^{p-2} -
  b^{p-2} \geq (p-2) (a - b) b^{p-3}$. Taking $a=\dist(x, \Sigma^*)$
  and $b = \dist(x, \Sigma)$ then yields
  \begin{align*}
    \dist(x,\Sigma^*)^{p-2} - \dist(x,\Sigma)^{p-2}
    &\geq (p-2) (\dist(x,\Sigma^*) - \dist(x,\Sigma))
      \dist(x,\Sigma)^{p-3}.
  \end{align*}
  Recognizing the right hand term as $\zeta(x)$,
  \eqref{eq:bound-where-zeta-appears} follows, completing the
  proof.
\end{proof}

\subsubsection{Refining \cref{lem:lower-bound-general} with
  \eqref{eqn:sigma-optimal} and \eqref{eqn:barycentrezero}}
With \cref{lem:lower-bound-general} in hand, we now separately
estimate the two integrals on the right side of \eqref{equationlb}.
The first estimate (\cref{lem:psi-integral}) is almost immediate,
while the second (\cref{lem:psi-lowerbound}) is slightly more
technical. Note, \cref{lem:psi-integral} requires the barycentre to be
trivial (see \eqref{eqn:barycentrezero}), while
\cref{lem:psi-lowerbound} requires \eqref{assum:zero-mu} and that
$\Sigma$ is optimal (see \eqref{eqn:sigma-optimal}).
\begin{lemma}\label{lem:psi-integral}
  Assume \eqref{eqn:barycentrezero} (that is,
  $\mathcal{B}_{\pi_{\Sigma}}\equiv 0$) and use the coordinate system
  with $\sigma^* = 0$ (see \eqref{coordinate-translation}) to construct
  $\Sigma^*$ as in \eqref{eqn:competitor}. Let $A \subset \Sigma$ be
  any subset of $\Sigma$. Then, for the $\psi$ defined in
  \cref{p=2lowerbound} we have
  \begin{equation}\label{firstterm}
    \begin{split}
      \int_{\mathbb{R}^d}\psi(x)\dist(x, \Sigma)^{p-2} d\mu(x) \ge
      & \int_{\pi_{\Sigma}^{-1}(A)}(-|\pi_{\Sigma}(x)|^2 + 2 \tau |x
        |_{\infty} - \tau^2)|x - \pi_{\Sigma}(x)|^{p-2}d\mu(x) \\
      &-\epsilon^2 \int_{\pi_{\Sigma}^{-1}(\Sigma \setminus
        A)}|\pi_{\Sigma}(x)|^2 |x - \pi_{\Sigma}(x)|^{p-2}d\mu(x).
    \end{split}
  \end{equation}
\end{lemma}

\begin{proof}
  By the definition of $\psi$ \eqref{eqn:psi-def}, writing $\dist(x,
  \Sigma) = \lvert x - \pi_\Sigma(x) \rvert$ we have
  \begin{align*}
    &  \int_{\mathbb{R}^d}\psi(x)\dist(x, \Sigma)^{p-2}d\mu(x)\\
    \ge
    & \int_{\pi_{\Sigma}^{-1}(A)}(-2\pi_{\Sigma}(x) \cdot (x -
      \pi_{\Sigma}(x)) - |\pi_{\Sigma}(x)|^2 + 2 \tau |x|_{\infty} -
      \tau^2)|x- \pi_{\Sigma}(x)|^{p-2} d\mu(x) \\
    &+ \int_{\pi_{\Sigma}^{-1}(\Sigma\setminus A)}(-2\epsilon
      \pi_{\Sigma}(x)\cdot (x - \pi_{\Sigma}(x)) - \epsilon^2
      |\pi_{\Sigma}(x)|^2)|x-\pi_{\Sigma}(x)|^{p-2}d\mu(x).
  \end{align*}
  Using the fact that the barycenter field
  $\mathcal{B}_{\pi_{\Sigma}}$ is trivial, we get
  \[
    \int_{\pi_{\Sigma}^{-1}(A)}\pi_{\Sigma}(x)\cdot (x -
    \pi_{\Sigma}(x))|x - \pi_{\Sigma}(x)|^{p-2}d\mu(x) =
    \int_{A}\sigma \cdot
    \mathcal{B}_{\pi_{\Sigma}}(\sigma)d\nu(\sigma) = 0,
  \]
  and similarly
  \[
    \int_{\pi_{\Sigma}^{-1}(\Sigma\setminus A)}\epsilon
    \pi_{\Sigma}(x)\cdot (x - \pi_{\Sigma}(x))|x -
    \pi_{\Sigma}(x)|^{p-2}d\mu(x) =0.
  \]
  Dropping these terms from the first equation then yields the
  desired result.
\end{proof}
Now, we find a lower bound on the term $\int_{\mathbb{R}^d} \psi(x)
\zeta(x) d\mu(x)$.

\begin{lemma}\label{lem:psi-lowerbound}
  Take \eqref{assum:zero-mu} and assume $\Sigma$ is optimal (see
  \eqref{eqn:sigma-optimal}). Fix $\epsilon>0$ and $\sigma^*\in
  \Sigma$, and taking the coordinate system
  \eqref{coordinate-translation}, construct $\Sigma^*$ according to
  \eqref{eqn:competitor}. For $\alpha$ as defined in
  \eqref{eqn:tau-alpha}, let
  \begin{align}\label{eqn:Mandc}
    M =  \diam(\hbox{the convex
    hull of $\supp \mu$})
    \qquad \hbox{and}  \ \ \   c=\max[M, \alpha].
  \end{align}
  Fix an arbitrary set $A \subset \Sigma$, and let
  \begin{align}\label{eqn:beta}
    \beta_A  = \sup_{x\in A} |x|
  \end{align}
  and
  \begin{align}\label{eq:kappa}
    \kappa_p=\kappa_p(\epsilon) = \Bigg\{ \begin{matrix}
      0, & p =2, \\
      (c\epsilon)^{p-2}, & 2 < p < 3. \\
      (p-2)c\epsilon M^{p-3}, & p \geq 3.
    \end{matrix}
  \end{align}
  Then for $\zeta$ defined in \cref{lem:lower-bound-general} we have
  \begin{align}\label{secondterm}
    &\int_{\mathbb{R}^d}\psi(x)\zeta(x) d\mu(x)\\\nonumber
    &\geq - \kappa_p (\epsilon)\left((2\beta_A M + \beta_A^2 + 2\tau M
      + \tau^2)\nu(A) + (2\epsilon M^2 + \epsilon^2 M^2)\nu(\Sigma
      \setminus A)\right).
  \end{align}
\end{lemma}
The decay rate of $\kappa_p(\epsilon)$ in $\epsilon$ will be used
crucially in the subsequent sections for the proof of
\cref{nontrivialbarycentre}.
\begin{proof}
  We proceed in two steps. In the first step, we will obtain
  $\kappa_p$ by bounding $\zeta(x)$, and then in the next step, we
  will bound $\int |\psi|\ d \mu$ to obtain the desired result.

  \textbf{Step 1.} Let $x \in \supp \mu$ be arbitrarily chosen. To
  estimate $\dist(x, \Sigma) - \dist(x, \Sigma^*)$, notice that
  $(1-\epsilon) \pi_\Sigma(x) \in \Sigma_\epsilon \subset \Sigma^*$,
  therefore,
  \[
    \dist(x, \Sigma^*) \le |x- (1-\epsilon)\pi_\Sigma (x) | \le |x
    -\pi_\Sigma (x)| + \epsilon|\pi_\Sigma (x)| \le \dist(x, \Sigma) +
    \epsilon|\pi_\Sigma (x)|.
  \]
  By \cref{thm:optimizers-in-chull}, $\Sigma \subseteq
  \mathrm{ConvexHull}(\supp \mu)$. Since $x \in
  \mathrm{ConvexHull}(\supp \mu)$ as well,
  \[
    \dist(x, \Sigma^*)- \dist(x, \Sigma) \le \epsilon|\pi_\Sigma
    (x)|\le \epsilon M.
  \]
  On the other hand, $\pi_{\Sigma^*} (x) \subset \Sigma_\epsilon \cup
  S_\tau$, and every $x \in S_\tau$ is $\tau$-close to $\sigma^* \in
  \Sigma$. So,
  \begin{align*}
    |\pi_{\Sigma^*}(x) - \pi_\Sigma (\pi_{\Sigma^*}(x))| \le
    \max[\tau, M\epsilon].
  \end{align*}
  From this we see
  \begin{align*}
    \dist (x, \Sigma)
    & \le |x-\pi_{\Sigma^*}(x)| + | \pi_{\Sigma^*}(x) -
      \pi_\Sigma(\pi_{\Sigma^*}(x))| \\
    &= \dist(x, \Sigma^*) +| \pi_{\Sigma^*}(x) -
      \pi_\Sigma(\pi_{\Sigma^*}(x))|\\
    & \le \dist(x, \Sigma^*) + \max[\tau, M\epsilon],
  \end{align*}
  whence
  \begin{align*}
    \dist (x, \Sigma) - \dist(x, \Sigma^*) \le \max[\tau, M\epsilon].
  \end{align*}
  In summary: Recalling $\tau = \alpha \epsilon$ and $c = \max[\tau,
  M\epsilon]$, we have shown
  \begin{align*}
    |\dist (x, \Sigma) - \dist(x, \Sigma^*)| \le \max[\tau, M\epsilon]=
    c \epsilon.
  \end{align*}
  Recalling the definition of $\zeta(x)$
  (\cref{lem:lower-bound-general}) we see
  \begin{equation}
    \zeta(x) \leq \kappa_p(\epsilon). \label{eq:zeta-kappa-ineq}
  \end{equation}

  \textbf{Step 2.} Now, we wish to control $\int |\psi|\ d\mu$. Continue
  assuming $x \in \supp \mu$. Then, by the definition of $\psi$
  \eqref{eqn:psi-def}, if $x \in \pi_{\Sigma}^{-1}(A)$ then
  \cref{thm:optimizers-in-chull} gives
  \begin{equation}
    |\psi(x)| \leq  2\beta_A M  +\beta_A^2 +2\tau M +
    \tau^2, \label{eq:psi-bmt-one}
  \end{equation}
  while if $x \in \pi_{\Sigma}^{-1}(\Sigma \setminus A)$ we get
  \begin{equation}
    |\psi(x)| \leq 2\epsilon M^2 + \epsilon^2
    M^2. \label{eq:psi-bmt-two}
  \end{equation}
  Recalling that $x \in \supp \mu$ was arbitrarily chosen, it follows
  that \eqref{eq:zeta-kappa-ineq}, \eqref{eq:psi-bmt-one}, and
  \eqref{eq:psi-bmt-two} hold $\mu$-a.e. Chaining them together, we get
  \begin{align*}
    &\int_{\mathbb{R}^d}\psi(x)\zeta(x) d\mu(x)\\
    &\geq -\kappa_p(\epsilon)\int_{\mathbb{R}^d}|\psi(x)|d\mu(x)\\
    &\geq - \kappa_p (\epsilon)\left((2\beta_A M + \beta_A^2 + 2\tau M
      + \tau^2)\nu(A) + (2\epsilon M^2 + \epsilon^2 M^2)\nu (\Sigma
      \setminus A)\right),
  \end{align*}
  as desired.
\end{proof}

\subsubsection{Concluding the first main step} Chaining together
\crefrange{lem:lower-bound-general}{lem:psi-lowerbound} we have the
following overall result.
\begin{corollary} \label{cor:step-one-summary}
  Let $p \geq 2$, and take assumptions \eqref{assum:zero-mu},
  \eqref{eqn:sigma-optimal}, and \eqref{eqn:barycentrezero}. Using
  \eqref{coordinate-translation}, construct the competitor $\Sigma^*$
  as in \eqref{eqn:competitor}. Then
  \begin{align*}
    & \mathscr{J}_p(\Sigma) - \mathscr{J}_p(\Sigma^*) \\
    & \ge \frac{p}{2} \int_{\pi_{\Sigma}^{-1}(A)}(-|\pi_{\Sigma}(x)|^2
      + 2 \tau |x |_{\infty} - \tau^2)|x -
      \pi_{\Sigma}(x)|^{p-2}d\mu(x) \\
    &\quad -\epsilon^2 \frac{p}{2} \int_{\pi_{\Sigma}^{-1}(\Sigma
      \setminus A)}|\pi_{\Sigma}(x)|^2 |x -
      \pi_{\Sigma}(x)|^{p-2}d\mu(x)\\
    & \quad - \frac{p}{2}\kappa_p(\epsilon)\left((2\beta_A M +
      \beta_A^2 + 2\tau M + \tau^2)\nu(A) + (2\epsilon M^2 +
      \epsilon^2 M^2)\nu(\Sigma \setminus A)\right).
  \end{align*}
\end{corollary}
This concludes the first main step of our roadmap. We now move to the
second: picking $\sigma^*$, $A$ in such a way that
\cref{cor:step-one-summary} yields a contradiction with the optimality
of $\Sigma$.

\subsection{Second step: favourably choosing $\sigma^*$ and
  $A$.}\label{sec:second-main-step} We want to find a choice of
$\sigma^*$ and $A$ which allows us to control both $\nu(A)$ and
$\beta_A$ in the inequality \eqref{secondterm}. In order to choose $A$
in the best possible way, we will consider separate cases, depending
on how $\nu$ behaves. To give very coarse intuition, we can think of
dividing these cases roughly into
\begin{align*}
  \hbox{``$\nu$ has an atom'' and ``$\nu$
  does not have an atom.''}
\end{align*}
However, it is necessary to be more careful than this, so we
introduce for $s, K \geq 0$ the sets $B_K^s$ defined in
\eqref{eq:badset}, which quantify concentration of $\nu$ around a
point. Loosely speaking, the $B_K^s$ represent points where $\nu$
has ``local dimension'' not greater than $s$.
\begin{equation}\label{eq:badset}
  B_K^s \coloneqq \left\{\sigma \in \Sigma \mid \limsup_{r\to 0}
    \frac{1}{r^s} \nu(B_r(\sigma)) > K\right\}.
\end{equation}

One may show that for $s=1$, $\nu (B^s_K) >0$ for a sufficiently small
$K>0$ due to the finite $\mathcal{H}^1$-measure of $\Sigma$. However,
for our proof of \cref{nontrivialbarycentre}, we are interested in the
range $0\le s <1$ for which $\nu(B^s_K)>0$ represents an unusual
concentration of measure, similar to having an atom (the case $s=0$).
So, the cases ``$\nu$ has an atom'' and ``$\nu$ does not have an
atom'' can be generalized as
\begin{align}\label{eqn:atom-or-not}
  \hbox{ whether for some $K>0$, $0\le s <1$, (i) $\nu(B^s_K)>0$ or
  (ii) $\nu(B^s_K)=0$.}
\end{align}

To treat case (i) of \eqref{eqn:atom-or-not} we prove
\cref{badsetbig}, which will be used in \cref{sec:finalsteps}. Here,
the rough idea is that if $\nu$ concentrates mass near a point
$\sigma^*$, the cross shape $S_\tau$ used to construct $\Sigma^*$ in
\eqref{eqn:competitor} will improve (decrease) the average distance to
$\mu$. Importantly, the barycentre field being trivial
\eqref{eqn:barycentrezero} is used to ensure that the replacement of
$\Sigma$ by $\Sigma_\epsilon$ in the construction of $\Sigma^*$ will
have comparatively negligible impact. See also
\cref{rmk:atom-nontrivialbary} for the case $s=0$, i.e.\ $\nu$ has an
atom.

\begin{lemma}\label{badsetbig}
  Assume \eqref{assum:zero-mu}, \eqref{eqn:sigma-optimal} and
  \eqref{eqn:barycentrezero}. Recall $\alpha$ from
  \eqref{eqn:tau-alpha}, $M$ from \eqref{eqn:Mandc}, and $\kappa_p$
  from \eqref{eq:kappa}. Let $0 \leq s < 1$ and $K \geq 0$, and recall
  $B_K^s$ from \eqref{eq:badset}. Suppose $$\nu(B_K^s) > 0.$$ Then,
  there exists a constant $0 < C< \infty$ such that
  for each $\delta>0$, there exist $0< \epsilon<\delta$ and $\sigma^*
  \in \Sigma$ such that the associated $\Sigma^*$
  \eqref{eqn:competitor} satisfies
  \begin{align}\label{eq:better-competitor-concentration-case}
    \mathscr{J}_p(\Sigma) - \mathscr{J}_p(\Sigma^*) \geq
    \frac{p}{2} \left( K\epsilon^{1 + s}\eta_0(\epsilon) -
    \epsilon\kappa_p(\epsilon) \eta_1(\epsilon) -  \epsilon^4 M^{p-2}
    \right),
  \end{align}
  where
  \[
    \lim_{\epsilon \to 0^+} \eta_0(\epsilon) = 2C\alpha > 0
    \qquad \text{ and } \qquad
    \lim_{\varepsilon \to 0^+} \eta_1(\epsilon) = 2M(1 + \alpha + M) >
    0.
  \]
\end{lemma}
Before we give a proof, we remark that $\kappa_p (\epsilon)$ decays at
a certain power (depending on $p$), therefore with the proper choice
of $0 \leq s < 1$, the right-hand side of
\eqref{eq:better-competitor-concentration-case} gives a positive value
for sufficiently small $\epsilon$. This will be used in Section
\ref{sec:finalsteps} for our contradiction in the case $\nu(B_K^s) >
0$.

\begin{proof}
  Since we took \eqref{assum:zero-mu}, \eqref{eqn:sigma-optimal}, and
  \eqref{eqn:barycentrezero}, we have bound from
  \cref{cor:step-one-summary}. From this we will extract the estimate
  \eqref{eq:better-competitor-concentration-case} as follows.

  In Steps 1 and 2 we construct for general $\delta > 0$ a
  distinguished $\sigma_\delta \in \Sigma$ where a certain integral
  quantity is lower bounded. In Step 3, we fix $\delta > 0$ and define
  an $\epsilon \leq \delta$ via the construction of $\sigma_\delta$.
  We then take $\sigma^* = \sigma_\delta$ and define the associated
  set $A$, and show two more easy bounds. Then, in Steps 4 and 5 we
  use these choices of $\sigma^*$ and $A$ to separately estimate the
  first two terms and the third term from \cref{cor:step-one-summary}
  (respectively, the estimates from
  \cref{lem:psi-integral,lem:psi-lowerbound}). Combining these
  estimates yields \eqref{eq:better-competitor-concentration-case}.

  \textbf{Step 1.} Define
  \[
    C \coloneqq \frac{1}{\nu(B_K^s)} \int_{\pi_{\Sigma}^{-1}(B_K^s)}
    |x-\pi_\Sigma (x)|_{\infty} |x-\pi_{\Sigma}(x)|^{p-2} d\mu(x).
  \]
  By $\nu(B_k^s) > 0$
  we see $C$ is well-defined, and by $\mu(\Sigma) = 0$
  \eqref{assum:zero-mu} we get $C > 0$. Thus, by construction,
  \[
    C>0 \text{ depends only on $\mu$, $\Sigma$, $K$, and $s$.}
  \]

  \textbf{Step 2.} We now use Vitali's covering argument to find a small
  ball where the mass of $\nu$ is concentrated. For each $\delta>0$,
  let
  \[
    \mathcal{V}^{\delta} = \{B_r(\sigma) \subseteq \mathbb R^d
    \mid \sigma \in \Sigma, \ 0 < r < \delta, \text{ and
    }\nu(B_r(\sigma)) \geq K r^s\}.
  \]
  Then, $\mathcal{V}^{\delta}$ is a Vitali covering of $B^s_K$. So, by
  the Vitali covering theorem for Radon measures
  \cite{Mattila95}*{Theorem 2.8}, we may find a countable disjoint
  subcollection $\{ U^\delta_j\}$ such that
  \begin{equation}
    \label{eq:vitalicover}
    \nu(B_K^s \setminus \bigcup_{j \in J}U^\delta_j)= 0 \hbox{, so
      by disjointness }
    \sum_{j \in J}\nu(U^\delta_j) = \nu(B_K^s).
  \end{equation}
  By \eqref{eq:vitalicover} we may write
  \begin{align*}
    & \sum_{j \in J} \int_{\pi_{\Sigma}^{-1}(U^\delta_j)}
      |x-\pi_\Sigma (x)|_{\infty} |x-\pi_{\Sigma}(x)|^{p-2}d\mu(x) \\
    & = \int_{\pi_{\Sigma}^{-1}(B_K^s)} |x-\pi_\Sigma(x)|_{\infty}
      |x-\pi_{\Sigma}(x)|^{p-2} d\mu(x)\\
    & = C \nu(B_K^s)\\
    & = C\sum_{j \in J}\nu(U^\delta_j),
  \end{align*}
  thus
  there exists at least one $k \in J$ such that
  \begin{align}\label{eqn:Vitali-Uk}
    \int_{\pi_{\Sigma}^{-1}(U^\delta_k)} |x-\pi_\Sigma(x)|_{\infty}
    |x-\pi_{\Sigma}(x)|^{p-2}d\mu(x) \geq C \nu(U^\delta_k).
  \end{align}
  By construction of $\mathcal{V}^{\delta}$, we have
  \begin{align}\label{eqn:r-delta}
    \hbox{$U^\delta_k = B_{r}(\sigma_{\delta})$ for some $0 < r <
    \delta$ and $\sigma_{\delta} \in B_{K}^s$.}
  \end{align}
  Now, note that for each $x \in \pi_\Sigma^{-1}(B_r(\sigma_d))$ we have
  $|\pi_\Sigma(x) - \sigma_\delta|_\infty \le r$. Then by the reverse
  triangle inequality,
  \begin{equation}
    \label{eq:proto-cr-bound}
    \lvert x - \sigma_\delta \rvert_\infty \geq \lvert x -
    \pi_\Sigma(x) \rvert_\infty - \lvert \pi_\Sigma(x) - \sigma_\delta
    \rvert_\infty \geq \lvert x - \pi_\Sigma(x)
    \rvert - r.
  \end{equation}
  Combining \eqref{eqn:Vitali-Uk}, \eqref{eqn:r-delta}, and
  \eqref{eq:proto-cr-bound}, we get
  \begin{align}\label{eqn:with-sigma-delta}
    \int_{\pi_{\Sigma}^{-1}(B_r(\sigma_\delta))}
    |x-\sigma_\delta|_{\infty} |x-\pi_{\Sigma}(x)|^{p-2}d\mu(x) \geq (C
    - r M^{p-2}) \nu(B_r(\sigma_\delta)).
  \end{align}

  \textbf{Step 3.} We now define $\epsilon$, $\sigma^*$, and $A$, and
  establish two easy bounds. Fix $\delta>0$, and let $\sigma_\delta$
  and $B_r(\sigma_\delta)$ be defined as in \eqref{eqn:r-delta}. Take
  \[
    \epsilon \coloneqq r,
  \]
  and note $\epsilon = r \leq \delta$. Also take $\sigma^* =
  \sigma_\delta$. For the remainder of the proof, by
  \eqref{coordinate-translation} we may use the coordinate system with
  \[
    \sigma^* = \sigma_\delta = 0.
  \]
  Note that this choice yields the simplification $\lvert x -
  \sigma_\delta \rvert_\infty = \lvert x \rvert_\infty$ in
  \eqref{eqn:with-sigma-delta}, as well as the simplification
  $B_r(\sigma_\delta) = B_\epsilon(0)$. In light of the latter, let
  \[
    A = B_\epsilon(0) \cap \Sigma.
  \]
  Note that $A \subseteq \Sigma$, as required for applying
  \cref{lem:psi-integral,lem:psi-lowerbound}. Also observe that for all
  $x \in \pi_\Sigma^{-1}(A)$, we trivially have $\pi_\Sigma(x) \in A$,
  and so
  \begin{equation}
    \lvert \pi_\Sigma(x) \rvert
    \leq
    \epsilon. \label{eq:eps-bound}
  \end{equation}
  Finally, since we took the hypotheses \eqref{assum:zero-mu} and
  \eqref{eqn:sigma-optimal}, \cref{thm:optimizers-in-chull} yields that
  for $\mu$-a.e.\ $x \in \mathbb R^d$,
  \begin{equation}
    \label{eq:M-bound}
    \lvert x - \pi_\Sigma(x) \rvert \leq M.
  \end{equation}

  Having defined $\epsilon$, $\sigma^*$, and $A$, we now combine the
  bounds we have derived so far to refine the estimates of
  \cref{lem:psi-integral,lem:psi-lowerbound} (respectively, the first
  two terms and the third term in \cref{cor:step-one-summary}).

  \textbf{Step 4.} Since we assumed $\pi_{\Sigma}$ has trivial
  barycentre field (see \eqref{eqn:barycentrezero}) and took the
  coordinate system \eqref{coordinate-translation}, we get
  \eqref{firstterm} from \cref{lem:psi-integral}. Chaining it together
  with \eqref{eq:eps-bound} and \eqref{eq:M-bound} yields
  \begin{align}
    & \int_{\mathbb{R}^d}\psi(x)\dist(x, \Sigma)^{p-2} d\mu(x) \nonumber
    \\
    & \geq - \epsilon^2 M^{p-2} \nu(A) +
      \bigg(\int_{\pi_{\Sigma}^{-1}(A)}2\tau |x
      |_{\infty}|x - \pi_{\Sigma}(x)|^{p-2}d\mu(x)\bigg) -
      \tau^2M^{p-2}\nu(A) \nonumber \\
    & \qquad - \epsilon^4 M^{p-2} \nu(\Sigma \setminus A), \nonumber
      \shortintertext{upon which using \eqref{eqn:with-sigma-delta},
      $\tau = \epsilon \alpha$, and $\nu(\Sigma \setminus A) \leq 1$, we
      have}
    & \geq \epsilon ( - \epsilon M^{p-2} + 2\alpha (C -
      \epsilon M^{p-2}) - \alpha^2 \epsilon M^{p-2})\nu(A) - \epsilon^{4}
      M^{p-2}. \label{eq:step-4-initial-bound}
  \end{align}

  Define
  \[
    \eta_0(\epsilon) \coloneqq (2\alpha (C - \epsilon M^{p-2})- \epsilon
    M^{p-2} - \alpha^2 \epsilon M^{p-2}),
  \]
  and note $\eta_0(\epsilon) = 2 \alpha C + O(\epsilon)$. Next, recall
  that $A$ was defined via $B_\delta(\sigma_\delta) \in \mathcal
  V^\delta$, the Vitali cover of $B_K^s$, so
  \[
    \nu(A) \geq K\epsilon^s.
  \]
  Using this, \eqref{eq:step-4-initial-bound} gives
  \[
    \int_{\mathbb{R}^d}\psi(x)\dist(x, \Sigma)^{p-2}d\mu(x) \geq K
    \epsilon^{1 + s}\eta_0(\epsilon) - \epsilon^4 M^{p-2}.
  \]

  \medskip

  \textbf{Step 5.} Similarly, by \eqref{assum:zero-mu} and
  \eqref{eqn:barycentrezero} we get \eqref{secondterm} from
  \cref{lem:psi-lowerbound}. That is, defining $\beta = \sup_{x\in A}
  \lvert x \rvert$, we get
  \begin{align*}
    \int_{\mathbb R^d} \psi(x) \zeta(x) \, d\mu
    &\geq - \kappa_p(\epsilon) ((2\beta M + \beta^2 + 2\tau M +
      \tau^2)\nu(A) + (2\epsilon M^2 + \epsilon^2 M^2){\nu(\Sigma
      \setminus A)}).
      \shortintertext{By construction of $A$, we have $\beta \leq
      \epsilon$ (in particular, $-\beta \geq -\epsilon$). Using the fact
      that $\nu(A), \nu(\Sigma \setminus A) \le 1$ and substituting
      $\tau = \alpha \epsilon$, we obtain the further bound}
    & \ge -\kappa_p  (\epsilon)((2\epsilon M + \epsilon^2 + 2 \alpha
      \epsilon M + \alpha^2 \epsilon^2) + (2 \epsilon M^2 + \epsilon^2
      M^2))
      \shortintertext{whence grouping like terms yields}
    &= -\kappa_p (\epsilon) \epsilon (2M(1 + \alpha + M) + (1 +
      \alpha^2 + M^2)\epsilon).
  \end{align*}
  So, defining
  \[
    \eta_1(\epsilon) \coloneqq (2M(1 + \alpha + M) + (1 + \alpha^2 +
    M^2)\epsilon),
  \]
  we have $\eta_1(\epsilon) = 2M(1 + \alpha + M) + O(\epsilon)$, and
  \[
    \int_{\mathbb{R}^d}\psi(x)\zeta(x) \geq -\kappa_p
    (\epsilon)\epsilon \eta_1(\epsilon).
  \]

  \textbf{Step 6.} Combining Steps 4 and 5, \eqref{equationlb} (or
  equivalently, \cref{cor:step-one-summary}) gives
  \[
    \mathscr{J}_p(\Sigma) - \mathscr{J}_p(\Sigma^*) \geq \frac{p}{2}
    \left(K\epsilon^{1 + s}\eta_0(\epsilon) - \kappa_p (\epsilon)
      \epsilon\eta_1(\epsilon) - \epsilon^4 M^{p-2} \right),
  \]
  with $\eta_0$ and $\eta_1$ possessing the desired limiting
  behaviour.
\end{proof}
\begin{remark}[Existence of atom implies nontrivial barycentre
  field.]\label{rmk:atom-nontrivialbary}

  At this moment we can prove that the existence of an atom implies
  the barycentre field of an optimizer is nontrivial. Suppose to
  obtain a contradiction that the barycentre field is trivial and that
  there exists an atom $\sigma^* \in \Sigma$ for $\nu$. Taking $K =
  \frac{1}{2}\nu(\{\sigma^*\})$ and $s = 0$, we see \eqref{eq:badset}
  that $\sigma^* \in B_{K}^s$, and so $\nu(B_K^s) > 0$. Since
  $\kappa_p (\epsilon)= o({1})$, \cref{badsetbig} shows that, for
  $\delta$ sufficiently small, there exists $\Sigma^* \in
  \mathcal{S}_l$ such that
  \[
    \mathscr{J}_p(\Sigma) - \mathscr{J}_p(\Sigma^*) > 0,
  \]
  contradicting the optimality of $\Sigma$.
\end{remark}
The next lemma will be used in \cref{sec:finalsteps} for the case
where there is no concentration of $\nu$ around a point (see (ii) of
\eqref{eqn:atom-or-not}). This case is more difficult to handle than
the previous ``atomic'' case. A rough idea is to find a $\sigma^*$
among the ``nonatomic'' points of $\nu$ such that the integral of
$|x|_\infty$ over $B_r(\sigma^*)$ decays ``slowly'' in $r$, i.e.\
order $O(r)$. This integral corresponds to the positive term in the
right hand side of \eqref{firstterm}. The quantitative nonatomic
property of $\sigma^*$ will give a certain decay rate for
$\nu(B_r(\sigma^*))$, which will make the other, possibly-negative
terms in the expansion of $ \mathscr{J}_p(\Sigma) -
\mathscr{J}_p(\Sigma^*)$ decay faster. Carrying this out is subtle and
it results in the following statement:
\begin{lemma}\label{badsetnull}
  Assume \eqref{assum:zero-mu}, \eqref{eqn:sigma-optimal} and
  \eqref{eqn:barycentrezero}. Recall $\alpha$ from
  \eqref{eqn:tau-alpha}, $M$ from \eqref{eqn:Mandc}, and $\kappa_p$
  from \eqref{eq:kappa}. Let $0 < s < 1$ and $K \geq 0$, and recall
  $B_K^s$ from \eqref{eq:badset}. Suppose
  \[
    \nu(B_K^s) =0.
  \]
  Then, there exists a constant $C_1>0$ such that the following holds:
  For each $0 < q < 1$, and $\lambda > 0$ and $K_{\lambda} \coloneqq K
  + \lambda$, there exists $\epsilon_0=\epsilon_0(\lambda)>0$ such
  that for each $\epsilon \le \epsilon_0$, there is a choice of
  $\sigma^*$ such that using the coordinate system $\sigma^* = 0$
  \eqref{coordinate-translation}, the associated $\Sigma^*$
  \eqref{eqn:competitor} satisfies
  \begin{align}\label{eqn:zero-case-estimate}
    & \mathscr{J}_p(\Sigma) - \mathscr{J}_p(\Sigma^*) \\
    & \ge \frac{p}{2}\Big( \epsilon^{1+q} \eta_2(\epsilon) -{2
      \alpha}K_\lambda \epsilon^{q(2+s)} M^{p-2}  -
      \kappa_p(\epsilon)(K_{\lambda}\epsilon^{(1 + s)q}\eta_3(\epsilon)
      + 2 M^2 \epsilon) \Big) + O(\epsilon^2) \nonumber
  \end{align}
  where
  \begin{align*}
    \lim_{\epsilon\to 0^+} \eta_2(\epsilon)= \frac{C_1}{8d} > 0
    \qquad \hbox{ and } \quad \lim_{\epsilon\to 0^+}
    \eta_3(\epsilon)=2M.
  \end{align*}

\end{lemma}
\begin{proof}
  \textbf{Step 1.} Let
  \begin{align*}
    S = \Sigma \setminus B_K^s =\{\sigma \in \Sigma \mid \limsup_{r
    \to 0}
    \frac{\nu(B_r(\sigma))}{r^s} \le K\}.
  \end{align*}
  By our hypothesis $\nu(B_K^s) = 0$, we have
  \[
    \nu(S) > 0.
  \]
  Since we want to bound integrals whose integrands may behave poorly
  near $\Sigma$, it is useful to split $\supp \mu$ into a tubular
  neighbourhood of the set $\Sigma$ and its compliment.

  Let $\delta > 0$, and consider the tubular neighbourhood
  $N_{\delta}(\Sigma) = \bigcup_{\sigma \in
    \Sigma}B_{\delta}(\sigma)$. Since $\nu(S) > 0$ and $\mu({ \mathbb
    R^d \setminus}\Sigma) = 1$ \eqref{assum:zero-mu}, recalling the
  relation $\nu=(\pi_\Sigma)_\# \mu$, continuity from above shows
  there exists $\delta>0$ so that $\mu(\pi_\Sigma^{-1}(S) \setminus
  N_\delta(\Sigma)) > 0$. In particular,
  \[
    C_1 \coloneqq \frac{\delta^{p-1}}{\sqrt{d}}
    \mu(\pi_{\Sigma}^{-1}(S) \setminus N_{\delta}(\Sigma)) > 0.
  \]
  The role of the $\delta^{p-1}/\sqrt d$ will become clear in Step 2;
  at this moment simply notice that the constant $C_1$ depends only on
  $\delta$, $d$ $S$, $\Sigma$, and $\mu$.

  For our quantitative argument we define a quantitative version of
  $S$ as follows. For each $\lambda>0$ and $t>0$, we let
  \begin{align}
    S_\lambda^t \coloneqq \{ \sigma \in \Sigma \mid \nu(B_r(\sigma))
    \le (K+\lambda) r^s \hbox{ for all $0<r\le t$}\}.
  \end{align}
  Notice that with respect to the partial ordering $\subseteq$, for
  fixed $t$, $S_\lambda^t$ is monotonically \emph{increasing} in
  $\lambda$; inversely, for fixed $\lambda$, it is monotonically
  \emph{decreasing} in $t$. Also,
  \begin{align*}
    S= \bigcap_{\lambda >0}  \bigcup_{t>0}  S^t_\lambda.
  \end{align*}
  Recall that for a general function $f$, the preimage $f^{-1}$
  respects arbitrary intersections/unions. In particular, for an
  arbitrary collection $\{X_j\}_{j \in J}$,
  \begin{align*}
    \pi_\Sigma^{-1} \Big(\bigcap _j X_j\Big)= \bigcap_j
    \pi_\Sigma^{-1} (X_j) \qquad \text{and} \qquad \pi_\Sigma^{-1}
    \Big(\bigcup_j X_j\Big) = \bigcup_j \pi_\Sigma^{-1} (X_j).
  \end{align*}
  Thus $\displaystyle \pi_{\Sigma}^{-1}(S) = \Big(\bigcap_{\lambda >0}
  \bigcup_{t>0} \pi^{-1} ( S^t_\lambda) \Big)$, and so
  \begin{align*}
    \pi_{\Sigma}^{-1}(S)\setminus N_{\delta}(\Sigma)
    &= \Big(\bigcap_{\lambda >0} \bigcup_{t>0} \pi_\Sigma^{-1} (
      S^t_\lambda) \Big)\setminus N_{\delta}(\Sigma)\\
    &=\bigcap_{\lambda >0} \bigcup_{t>0} \left( \pi_\Sigma^{-1} (
      S^t_\lambda) \setminus N_{\delta}(\Sigma) \right).
  \end{align*}
  Recalling that preimages respect $\subseteq$, we see that
  $\pi^{-1}(S^t_\lambda) \setminus N_{\delta}(\Sigma)$ inherits the
  monotonicity properties of $S_\lambda^t$ (i.e.\ for fixed $t$,
  increasing in $\lambda$; for fixed $\lambda$, decreasing in $t$).
  So, for any $\lambda > 0$,
  \[
    \Big(\bigcup_{t > 0}\pi_{\Sigma}^{-1}(S_{\lambda}^t)\setminus
    N_{\delta}(\Sigma) \Big) \supseteq (\pi_{\Sigma}^{-1}(S)\setminus
    N_{\delta}(\Sigma)),
  \]
  thus using continuity from below for the measure $\mu$, we see that
  there is some $t = t(\lambda) > 0$ such that
  \[
    \frac{\delta^{p-1}}{\sqrt{d}}\mu(\pi_{\Sigma}^{-1}(S_\lambda^t)
    \setminus N_{\delta}(\Sigma)) >
    \frac{1}{2}\frac{\delta^{p-1}}{\sqrt{d}}\mu(\pi_{\Sigma}^{-1}(S)
    \setminus N_{\delta}(\Sigma)) = \frac{1}{2}C_{1}.
  \]
  Below, we will consider such $S_\lambda^t$. Notice that such
  $\lambda$ and $t$ are chosen depending only on $d, \delta, \mu, S$
  and $\Sigma$.

  \medskip

  \textbf{Step 2.} We claim that for each sufficiently small $r>0$,
  there exists a finite collection of points $\sigma_1, \dots,
  \sigma_N \in S_\lambda^t$ such that
  \begin{equation}
    \hbox{$ S_\lambda^t \subseteq \bigcup_{i=1}^N B_r(\sigma_i)$ and
      $r N \leq 4l$.} \label{eq:B-r-sigma-i}
  \end{equation}
  To that end, since $\Sigma \in \mathcal S_l$, we have $\mathcal
  H^1(\Sigma) = l < \infty$. Thus \cite{Falconer86}*{Exercise\ 3.5}
  there exists a $2l$-Lipschitz curve $\gamma: [0,1] \to \mathbb{R}^d$
  with $\Sigma \subseteq \mathrm{image}(\gamma)$. Select $s_1, \dots,
  s_N \in [0,1]$ such that $[0,1] \subset
  \bigcup_{i=1}^N(s_i-\frac{r}{4l}, s_i + \frac{r}{4l})$ and
  $\frac{r}{4l} N \leq 1$. For each $i$, take $\sigma_i =
  \gamma(s_i)$. Since each $\gamma(s_i - \frac{r}{4l}, s_i +
  \frac{r}{4l}) \subseteq B_{2l\frac{r}{4l}}(\gamma(s_i))$, we have
  the desired properties.

  As in the proof of \cref{badsetbig} we will utilize
  \eqref{equationlb}, \eqref{firstterm} and \eqref{secondterm}. To
  that end, we first derive a simple bound that we will use with
  \eqref{firstterm}. For each $\sigma \in \Sigma$, notice that for all
  $x \not \in N_\delta(\Sigma)$ we have $|x-\sigma|_\infty \ge
  \delta/\sqrt{d}$ and $\lvert x - \pi_\Sigma(x) \rvert \geq \delta$.
  Therefore,
  \begin{align*}
    & \sum_{i =1}^N \int_{\pi_{\Sigma}^{-1}(B_r(\sigma_i)\cap
      \Sigma)\setminus N_{\delta}(\Sigma)} |x-\pi_\Sigma (x)|_{\infty}
      |x - \pi_{\Sigma}(x)|^{p-2}d\mu(x) \\
    &\geq \frac{\delta^{p-1}}{\sqrt{d}} \sum_{i=1}^N
      \mu(\pi_{\Sigma}^{-1}(B_r(\sigma_i)\cap \Sigma)\setminus
      N_{\delta}(\Sigma)) \\
    & \geq \frac{\delta^{p-1}}{\sqrt{d}}
      \mu(\pi_{\Sigma}^{-1}(S_\lambda^t) \setminus N_{\delta}(\Sigma))
    \\
    & > \frac{1}{2}C_1.
  \end{align*}
  So, there is some $k \in \{1, \dots, N\}$ such that
  \begin{equation*}
    \int_{\pi_{\Sigma}^{-1}(B_r(\sigma_k) \cap \Sigma)} |x
    -\pi_\Sigma(x)|_{\infty} |x-\pi_{\Sigma}(x)|^{p-2} d\mu(x) \geq
    \frac{C_1}{2N} \geq \frac{C_1}{8l}r.
  \end{equation*}
  For all $x \in \pi_\Sigma^{-1}(B_r(\sigma_k) \cap \Sigma)$ we see
  $|\pi_\Sigma(x) - \sigma_k|_\infty \le r$, thus for such $x$,
  \[
    |x-\sigma_k|_\infty \ge |x-\pi_\Sigma(x)|_\infty-r.
  \]
  Since $\mu(\pi_\Sigma^{-1}(B_r(\sigma_k) \cap \Sigma)) =
  \nu(B_r(\sigma_k) \cap \Sigma) = \nu(B_r(\sigma_k))$, we thus obtain
  \begin{align}\label{eq:sigma-k-integral}
    \int_{\pi_{\Sigma}^{-1}(B_r(\sigma_k)\cap \Sigma)}|x
    -\sigma_k|_{\infty}|x-\pi_{\Sigma}(x)|^{p-2}d\mu(x) \geq
    \frac{C_1}{8l}r -r M^{p-2} \nu (B_r(\sigma_k)).
  \end{align}

  \medskip

  \textbf{Step 3.}
  Take
  \begin{align*}
    \hbox{$A = B_r(\sigma_k)\cap \Sigma$ and  $\sigma^* = \sigma_k.$}
  \end{align*}
  By \eqref{coordinate-translation}, we take the coordinate system
  with $\sigma^*=0$, whence $\beta_A$ from \eqref{eqn:beta} is
  \[
    \beta_A = \max_{x \in A} \lvert x \rvert = r.
  \]
  Furthermore, since $\sigma^* \in S_\lambda^t$, we have for
  $K_\lambda \coloneqq K+\lambda$,
  \begin{equation}
    \hbox{$\nu(B_r(\sigma^*)) \leq K_{\lambda}r^s$ for each $0 < r\le
      t$.} \label{eq:nu-B-r-K-rs}
  \end{equation}

  Now, for some $0 < q < 1$,
  \begin{align*}
    \hbox{let $\epsilon_0=t^{1/q}$, fix $\epsilon \le \epsilon_0$, and
    take $r = \epsilon^q$.}
  \end{align*}
  (Notice that $\sigma^*$ implicitly depends on $\epsilon$, as
  $\sigma^* = \sigma_k$ depends on $r$ via \eqref{eq:B-r-sigma-i}).

  \textbf{Step 4.} By our hypothesis \eqref{eqn:barycentrezero} and
  our choice of coordinate system, \cref{lem:psi-integral} gives
  \begin{align*}
    & \int_{\mathbb{R}^d}\psi(x)\dist(x, \Sigma)^{p-2}d\mu(x)\\
    & \ge
      \int_{\pi_{\Sigma}^{-1}(A)}(-|\pi_{\Sigma}(x)|^2 + 2 \tau |x
      |_{\infty} - \tau^2)|x - \pi_{\Sigma}(x)|^{p-2}d\mu(x) \\
    &\qquad -\epsilon^2 \int_{\pi_{\Sigma}^{-1}(\Sigma \setminus
      A)}|\pi_{\Sigma}(x)|^2 |x -
      \pi_{\Sigma}(x)|^{p-2}d\mu(x). \shortintertext{For all the
      terms except the $\lvert x \rvert_\infty$ part, proceeding as
      the beginning of Step 4 of the proof of \cref{badsetbig}, but
      here using $\lvert \pi_\Sigma(x) \rvert \leq r = \epsilon^q$ and
      using \eqref{eq:nu-B-r-K-rs} to bound $\nu(A)$ with $\leq
      K_\lambda \epsilon^{sq}$ gives}
    &\geq -(\epsilon^{2q} +
      \alpha^2 \epsilon^2)M^{p-2} K_\lambda \epsilon^{sq}
      + \bigg(\int_{\pi_{\Sigma}^{-1}(A)}2\tau |x
      |_{\infty}|x - \pi_{\Sigma}(x)|^{p-2}d\mu(x)\bigg) -
      \epsilon^{2 + 2q} M^{p-2}. \shortintertext{For the $\lvert x
      \rvert_\infty$ term, note that in our coordinate system we have
      $\lvert x \rvert_\infty = \lvert x - \sigma_k \rvert_\infty$, so
      applying \eqref{eq:sigma-k-integral} and rearranging terms we
      get}
    & \ge \epsilon^{1+q}\Big(2\alpha \frac{C_1}{8l} -2\alpha
      K_\lambda \epsilon^{sq}\Big) M^{p-2} - K_\lambda
      (\epsilon^{2q+sq} + \alpha^2 \epsilon^{2+sq}) M^{p-2}
      - \epsilon^{2+2q} M^{p-2}\\
    & = \epsilon^{1+q} \eta_2(\epsilon) -K_\lambda \epsilon^{q(2+s)}
      M^{p-2} + O(\epsilon^2),
  \end{align*}
  where (recalling that $\alpha = \frac{l}{2d}$; see
  \eqref{eqn:tau-alpha}) $\eta_2$ is given by
  \[
    \eta_2(\epsilon) \coloneqq \frac{C_1}{8d} -2\alpha K_\lambda
    \epsilon^{sq} M^{p-2}.
  \]
  Notice that since $q,s >0$, as $\epsilon \to 0$ we have
  $\displaystyle \eta_2(\epsilon)\to \frac{C_1}{8d} >0$.

  \textbf{Step 5.} Similarly, applying $\beta_A = r = \epsilon^q$,
  $\tau = \alpha \epsilon$, and $\nu(A) \leq K_\lambda r^s = K_\lambda
  \epsilon^{sq}$ to \eqref{secondterm}, we have
  \begin{align*}
    &\int_{\mathbb{R}^d}\psi(x)\zeta(x) d\mu(x)\\\nonumber
    &\geq - \kappa_p(\epsilon) \left((2\beta_A M + \beta_A^2 + 2\tau M +
      \tau^2)\nu(A) + (2\epsilon M^2 + \epsilon^2 M^2)\nu(\Sigma
      \setminus A)\right) \\
    &\geq -\kappa_p (\epsilon)\big((2\epsilon^q M + \epsilon^{2q} +
      2\alpha \epsilon M + \alpha^2 \epsilon^2) K_\lambda \epsilon^{sq}
      + (2\epsilon M^2 + \epsilon^2 M^2)\big) \\
    & = -\kappa_p(\epsilon) (K_\lambda \epsilon^{(1 + s)q}
      \eta_3(\epsilon) + 2\epsilon M^2 + \epsilon^2 M^2)
  \end{align*}
  where
  \[
    \eta_3(\epsilon) \coloneqq 2M + \epsilon^q + 2 \alpha \epsilon^{1
      - q}M + \alpha^2 \epsilon^{2 - q}.
  \]
  Since $0 < q < 1$ we see that $\eta_3(\epsilon)\to 2M$ as $\epsilon
  \to 0$.

  \textbf{Step 6.} Finally, combining the bounds from Steps 4 and 5 with
  \eqref{equationlb} (equivalently, \cref{cor:step-one-summary}), we
  have
  \begin{align*}
    & \mathscr{J}_p(\Sigma) - \mathscr{J}_p(\Sigma^*)\\
    & \ge \frac{p}{2}\Big( \epsilon^{1+q} \eta_2(\epsilon) -K_\lambda
      \epsilon^{q(2+s)} M^{p-2}- \kappa_p(\epsilon)
      (K_{\lambda}\epsilon^{(1 + s)q} \eta_3(\epsilon) + 2 M^2 \epsilon)
      \Big) + O(\epsilon^2),
  \end{align*}
  where $\eta_2(\epsilon)$, $\eta_3(\epsilon)$ have the desired limiting
  behaviour.
\end{proof}

This concludes the second main step of the proof.

\subsection{Final steps of the proof{: deriving a
    contradiction}} \label{sec:finalsteps}

By combining \cref{badsetbig,badsetnull}, we can now finally establish
that under the assumption \eqref{assum:zero-mu}, there is a
contradiction between \eqref{eqn:sigma-optimal} (``$\Sigma$ is
optimal'') and \eqref{eqn:barycentrezero} (``$\pi_\Sigma$ has trivial
barycentre field'').

We proceed by casework on $p$, separately treating the regimes $p=2$,
$p \geq 3$, and $\frac{3 + \sqrt 5}{2} < p < 3$. Of course, we
conjecture the result should hold in the intermediate range $2 < p
\leq \frac{3 + \sqrt 5}{2}$ as well, but it seems a different argument
would be necessary for this. In any case, for each of these $p$
regimes we prove the result by separately considering (for some $K
\geq 0$ and $s \geq 0$) the subcases $\nu(B_K^s) > 0$ and $\nu(B_K^s)
= 0$.

\medskip

\textbf{Case 1 ($p=2$).} First, suppose $p = 2$. Then, from
\eqref{eq:kappa}, $\kappa_p = 0$. Fix $s = \frac{2}{3}$, and let $K >
0$ be arbitrary. Consider $B_{K}^s$.

\emph{Subcase 1.1:} If $\nu(B_K^s) >0$, then by \cref{badsetbig}
we may find a sufficiently small choice of $\epsilon$ and a
corresponding choice of $\sigma^*$ so that
\[
  \mathscr{J}_p(\Sigma) - \mathscr{J}_p(\Sigma^*) \geq K
  \epsilon^{5/3}\eta_0(\epsilon) - \epsilon^4 M^{p-2} > 0,
\]
as desired.

\emph{Subcase 1.2:} If $\nu(B_K^s) = 0$, then taking $q =
\frac{7}{8}$ and $\lambda> 0$ in \cref{badsetnull}, for all
sufficiently small $\epsilon$ there is a choice of $\sigma^*$ so that
\begin{align*}
  \mathscr{J}_p(\Sigma) - \mathscr{J}_p(\Sigma^*)
  & \ge \frac{p}{2} \Big(\frac{C_1}{8d} \epsilon^{15/8} -
    2\alpha K_{\lambda} \epsilon^{1+\frac{35}{24}} M^{p-2}
    \Big) +O(\epsilon^2).
\end{align*}
As $\epsilon \to 0$, the $\epsilon^{15/8}$ term dominates, and thus
there is a choice of $\epsilon$ for which $\mathscr{J}_p(\Sigma) -
\mathscr{J}_p(\Sigma^*) > 0$, as desired.

\medskip

\textbf{Case 2 ($p \ge 3$).} Next, suppose $p \geq 3$. Then, from
\eqref{eq:kappa} $\kappa_p = a\epsilon$ for some constant $a$
depending only on $\mu$, $p$, and $\alpha$. As before, fix $s =
\frac{2}{3}$ and let $K > 0$ be arbitrary.

\emph{Subcase 2.1:} If $\nu(B_K^s) >0$, then by \cref{badsetbig}
we may find a sufficiently small choice of $\epsilon$ and a
corresponding choice of $\sigma^*$ so that
\[
  \mathscr{J}_p(\Sigma) - \mathscr{J}_p(\Sigma^*) \geq \frac{p}{2}
  \Big( K \epsilon^{1+2/3}\eta_0(\epsilon) - a\epsilon^2
  \eta_1(\epsilon) \Big)> 0,
\]
as desired.

\emph{Subcase 2.2:} If $\nu(B_K^s) = 0$, then as in Subcase 1.2 take
$q = \frac{7}{8}$ and $\lambda > 0$ in \cref{badsetnull}. The only
difference in this case is that we have $\kappa_p = a \epsilon$ rather
than $\kappa_p \equiv 0$; however, since $(1+s)q = \frac{35}{24} > 1$,
the term $-\kappa_p(\epsilon)(K_{\lambda}\epsilon^{(1 +
  s)q}\eta_3(\epsilon) + 2 M^2 \epsilon)$ has order $\epsilon^2$, and
so can be ignored safely. So we get $\mathscr J_p(\Sigma) - \mathscr
J_p(\Sigma^*) > 0$ in this case as well, as desired.

\medskip

\textbf{Case 3 $({\frac{3 + \sqrt 5}{2}} < p < 3)$.} This case is more
difficult, so for organizational reasons we shall state our two
subcases as lemmata. We initially proceed with the more general case
$2 < p < 3$ and then show the further restriction $p > \frac{3 + \sqrt
  5}{2}$ arises naturally from our bounds.

\begin{lemma}[Subcase 3.1]\label{ssmall}
  Suppose $2 < p < 3$. Then for any
  \[
    0 \leq s < p-2,
  \]
  if $\nu(B_K^s) > 0$ for some $K > 0$, then there exists some
  $\Sigma^* \in \mathcal{S}_l$ so that $\mathscr{J}_p(\Sigma) -
  \mathscr{J}_p(\Sigma^*) > 0$.
\end{lemma}

\begin{proof}
  Since $2< p < 3$, from \eqref{eq:kappa} we have $\kappa_p =
  a\epsilon^{p-2}$ for some $a$ depending only on $\mu$, $p$, and
  $\alpha$. By \cref{badsetbig}, we may find arbitrarily small choices
  of $\epsilon$ and a corresponding choice of $\sigma^*$ so that
  \[
    \mathscr{J}_p(\Sigma) - \mathscr{J}_p(\Sigma^*) \geq \frac{p}{2}
    \Big(K \epsilon^{1 + s}\eta_0(\epsilon) -
    a\epsilon^{1+(p-2)}\eta_1(\epsilon) - \epsilon^4 M^{p-2} \Big).
  \]
  Since $s < p-2 < 1$, taking $\epsilon$ sufficiently small yields
  $\mathscr{J}_p(\Sigma) - \mathscr{J}_p(\Sigma^*) > 0$, as desired.
\end{proof}

\begin{lemma}[Subcase 3.2]\label{sbig-new}
  Suppose that $2 < p < 3$. Then for any $s>0$ with
  \[
    \frac{1}{s+1}> p-2,
  \]
  if $\nu(B_K^s) = 0$ for some $K>0$, then there exists some
  $\Sigma^* \in \mathcal{S}_l$ so that $\mathscr{J}_p(\Sigma) -
  \mathscr{J}_p(\Sigma^*) > 0$.
\end{lemma}
\begin{proof}
  We will be using \eqref{eqn:zero-case-estimate} in
  \cref{badsetnull}. Since $2 < p < 3$, from \eqref{eq:kappa} we have
  $\kappa_p = a \epsilon^{p-2}$ for some $a$ depending only on $\mu$,
  $p$, and $\alpha$. Fix $0 < q < 1$, $s>0,$ and take $\lambda > 0$ in
  \cref{badsetnull} and $K_{\lambda} \coloneqq K + \lambda$. Observe
  that since the $\epsilon^{1+q}$ term is positive and since $1 + q <
  2$, for small $\epsilon$ we may ignore the higher-order
  $O(\epsilon^2)$ terms on the right hand side of
  \eqref{eqn:zero-case-estimate}, leaving
  \begin{align*}
    \epsilon^{1+q}\eta_2(\epsilon)
    -{2\alpha}K_\lambda \epsilon^{q(2+s)} M^{p-2} - a
    \epsilon^{p-2} (K_{\lambda}\epsilon^{(1 + s)q}\eta_3(\epsilon) + 2
    M^2 \epsilon).
  \end{align*}
  From this, recalling the limiting behaviour of $\eta_2, \eta_3$, we
  see that there exist constants $c_1, c_2 > 0$ such that for all
  $\epsilon$ sufficiently small,
  \begin{align*}
    \mathscr J_p(\Sigma) - \mathscr J_p(\Sigma^*)
    & \ge c_1 \epsilon^{1+q} - c_2(\epsilon^{2q+qs} +
      \epsilon^{p-2+ q+sq} + \epsilon^{p-2+1}).
  \end{align*}
  For this to be positive as $\epsilon \to 0$, it is sufficient to
  have
  \begin{align*}
    1+q < \min[ 2q+qs, \ p-2+q+qs, \ p-2+1].
  \end{align*}
  That is,
  \begin{enumerate}[label=(\roman*)]
    \item $1+q< 2q+qs$, and
    \item $1+q < p-2+q+qs$, and
    \item $1+q < p-2+1$.
  \end{enumerate}
  Condition (i) is equivalent to $\frac{1}{1+s} < q$; (ii) is
  equivalent to $\frac{3-p}{s} < q$; and (iii) is equivalent to $q<
  p-2$. Thus it suffices to find $s > 0$ such that
  \begin{align*}
    \frac{1}{1+s} < p-2 \quad \hbox{and} \quad \frac{3-p}{s} < p-2.
  \end{align*}
  Observe that the second inequality is equivalent to $(p-2)s > (3-p)
  = 1 - (p-2)$, whence rearranging gives $1/(1+s) < p-2$ again. So
  these requirements are redundant, and since $2 < p < 3$, picking any
  $s > \frac{3-p}{p-2}$ suffices.

  In summary{, for any $s > \frac{3-p}{p-2}$, we can choose a $q$
    with $\frac{1}{1+s} < q < p-2$, whence taking $\epsilon$
    sufficiently small in \eqref{eqn:zero-case-estimate} yields
    $\mathscr J_p(\Sigma) - \mathscr J_p(\Sigma^*) > 0$ as desired.}
\end{proof}

Finally, combining \cref{ssmall,sbig-new}, we see that if there exists
some $s$ with $\frac{3-p}{p-2} < s < p-2$, then we can find $\Sigma^*
\in \mathcal{S}_l$ such that $\mathscr{J}_p(\Sigma) -
\mathscr{J}_p(\Sigma^*) > 0$. Existence of such $s$ is equivalent to
\begin{align*}
  \frac{3-p}{p-2} < p-2, \quad \text{ or equivalently,} \quad 0 <
  p^2 - 3p + 1.
\end{align*}
Together with $2 < p < 3$, this gives $\frac{3 + \sqrt 5}{2} < p < 3$.

Combining cases 1-3, we see that in each of these cases $\mathscr
J_p(\Sigma) - \mathscr J_p(\Sigma^*) > 0$, a contradiction to
optimality of $\Sigma$ \eqref{eqn:sigma-optimal}. So
\eqref{eqn:barycentrezero} cannot hold. This completes the proof of
\cref{nontrivialbarycentre}. \qed

\section*{Acknowledgements}

The first author would like to thank Robert McCann for his valuable
comments on the presentation of this paper. The second author would
like to thank Nitya Gadhiwala for helpful discussions regarding
\cref{thm:optimizers-in-chull}, and Philip D. Loewen for identifying
some pernicious typos.

\bibliographystyle{amsplain}
\bibliography{bib}

\end{document}